\documentclass[a4paper,12pt]{article}

\usepackage{amsmath,mathrsfs, eufrak,amssymb,amsthm,latexsym,booktabs,enumitem,color,multirow,url,verbatim,microtype,tabularx,authblk,calligra,graphicx,subcaption}

\usepackage[dvipsnames]{xcolor}
\usepackage{tikz}
\usetikzlibrary{patterns,arrows}

\numberwithin{equation}{section}

\def\e{\varepsilon}

\def\C{\mathbb{C}}
\def\R{\mathbb{R}}

\def\Z{\mathbb{Z}}
\def\N{\mathbb{N}}
\def\S{\mathbb{S}}
\def\D{\mathscr{D}}

\def\H{\mathfrak{H}}
\def\MT{\mathcal{MT}}
\def\V{\mathcal{V}}
\def\A{\mathcal{A}}
\def\B{\mathcal{B}}

\def\tensor{\otimes}

\def\til{~}
\definecolor{myyellow}{rgb}{0.9290 0.6940 0.1250}
\definecolor{myorange}{rgb}{0.8500 0.3250 0.0980}
\definecolor{myred}{rgb}{0.6350 0.0780 0.1840}
\definecolor{mygreen}{rgb}{0.4660 0.6740 0.1880}
\definecolor{mycyan}{rgb}{0.3010 0.7450 0.9330}
\definecolor{myblue}{rgb}{0 0.4470 0.7410}
\definecolor{mypurple}{rgb}{0.4940 0.1840 0.5560}
\newcommand{\n}[1]{\left\lVert#1\right\rVert}

\DeclareMathOperator*{\sgn}{sgn}
\DeclareMathOperator*{\dom}{dom}
\DeclareMathOperator*{\dist}{dist}

\newtheorem{theorem}{Theorem}

\newtheorem{lemma}{Lemma}
\newtheorem{corollary}{Corollary}

\theoremstyle{definition}

\newtheorem{remark}{Remark}[section]
\newtheorem{hypothesis}{Assumption}

\newtheorem*{rigidityassumptions*}{Rigidity Assumptions}

\title{Keller-type bounds for Dirac operators \\ perturbed by rigid potentials}

\author{
	Haruya Mizutani\thanks{Department of Mathematics, Graduate School of Science, Osaka University, Toyonaka, Osaka 560-0043, Japan. E-mail address: {haruya@math.sci.osaka-u.ac.jp}} 
	\;and 
	Nico M. Schiavone\thanks{Department of Mathematics ``Guido Castelnuovo'', Sapienza Universit\`{a} di Roma,
		Piazzale Aldo Moro\til5, 00185 Rome, Italy.
		E-mail address: {schiavone@mat.uniroma1.it}}
}

\date{\vspace{-6ex}}

\begin{document}
\maketitle

\begin{center}
	\footnotesize
	
	\begin{tabularx}{0.7\textwidth}{lX}
		\begin{tabular}{@{}l@{}}
		\end{tabular}
		& 
		\begin{tabular}{@{}l@{}}
			non-selfadjoint,
			localization of eigenvalues,
			Keller-type bound
			\\
			Dirac operator, Birman-Schwinger principle 
		\end{tabular}
		\\[0.35cm]
		{\bf MSC2020:}
		&
		primary 35P15, 35J99, 47A10, 47F05, 81Q12
	\end{tabularx}

\end{center}


\quad


\begin{abstract}
	In this paper we are interested in generalizing Keller-type eigenvalue estimates for the non-selfadjoint Schr\"odinger operator to the Dirac operator, imposing some suitable rigidity conditions on the matricial structure of the potential, without necessarily requiring the smallness of its norm.
\end{abstract}


\tableofcontents


\section{Introduction}

The last three decades have seen a flood of interest in the study of non-selfadjoint operators in Quantum Mechanics, both from a physical point of view, due to the possibility of a new quantomechanic formulation involving quasi-selfadjoint or $\mathcal{PT}$-symmetric observables, and from a mathematical one, since the sacrifice of the selfadjoint property turns out in a lack of mathematical tools, which makes this topic challenging.

In particular, huge attention is paid to the spectral properties of non-selfadjoint operators and to the so-called Keller-type inequalities, id est bounds on the eigenvalues in terms of norms of the potential. Especially in the case of the Schr\"odinger operator, they can be referred to as well as Lieb-Thirring-type inequalities, since they constitute somewhat the counterpart of the celebrated inequalities for the {selfadjoint} Schr\"odinger operator, exploited by Lieb and Thirring in the \lq 70 of the last century to prove the stability of matter (we just cite the monograph \cite{LiebSeiringer-book} for an academic treatment of the subject).

The first appearance of a Keller-type inequality for the non-selfadjoint Schr\"odinger operator $-\Delta+\V$, where the potential $\V$ is a complex-valued function, is due to Abramov, Aslanyan and Davies \cite{AbramovAslanyanDavies01} which observed that the bound
\begin{equation}\label{eq:AAD}
	|z|^{1/2} \le \frac{1}{2} \n{\V}_{L^1}
\end{equation}
holds, in dimension $n=1$, for any eigenvalue $z \in \sigma_p(-\Delta+\V)$, and the constant is sharp.

In view of this result, Laptev and Safronov \cite{LaptevSafronov09} conjectured that the eigenvalues localization bound $|z|^{\gamma} \le D_{\gamma,n} \n{\V}_{L^{\gamma+n/2}}^{\gamma+n/2}$ should hold for any $0<\gamma\le n/2$ and some constant $D_{\gamma,n}>0$. 
Frank \cite{Frank11} proved the conjecture to be true for $0<\gamma\le 1/2$, and later Frank and Simon \cite{FrankSimon17} extended the range up to the one suggested by Laptev and Safronov under radial symmetry assumptions. Explicitly, in dimension $n\ge2$ the eigenvalues of $-\Delta+\V$ satisfy the estimates 
\begin{equation}\label{eq:schrodinger-enclosures}
	|z|^{\gamma} \le D_{\gamma,n} 
	\left\{
	\begin{aligned}
	&\n{\V}_{L^{\gamma+n/2}}^{\gamma+n/2}
	&&\text{for $0<\gamma\le\frac{1}{2}$,}
	\\
	&\n{\V}_{L_\rho^{\gamma+n/2} L_\theta^\infty}^{\gamma+n/2}
	&&\text{for $\frac{1}{2}<\gamma< \frac{n}{2}$,}
	\\
	&\n{\V}_{L_\rho^{n,1} L_\theta^\infty}^{n}
	&&\text{for $\gamma=\frac{n}{2}$,}
	\end{aligned}
	\right.
\end{equation}
where the positive constant $D_{\gamma,n}$ is independent of $z$ and $\V$ and where the radial-angular spaces $L_{\rho}^p L_{\theta}^s$ and $L_{\rho}^{p,q}L_{\theta}^s$ are defined as
\begin{equation}\label{eq:radial-angular-spaces}
	\begin{aligned}
		L_{\rho}^p L_{\theta}^s
		:=&\,
		L^{p}(\R_+, r^{n-1}dr ; L^s(\S^{n-1}))
		\\
		L_{\rho}^{p,q}L_{\theta}^s
		:=&\,
		L^{p,q}(\R_+, r^{n-1}dr ; L^s(\S^{n-1}))
	\end{aligned}
\end{equation}
being $L^{p,q}$ the Lorentz spaces and $\S^{n-1}$
the $n$-dimensional unit spherical surface.
In the case $1 \le p,q<\infty$, the respective norms are explicitly given by
\begin{equation}\label{eq:radial-angular}
	\begin{aligned}
		\n{f}_{L_{\rho}^{p}L_{\theta}^s}
		&:=
		\left( \int_0^\infty 
		\n{f(r\,\cdot)}_{L^s(\S^{n-1})}^p
		r^{n-1} dr\right)^{1/p}
		\\
		\n{f}_{L_{\rho}^{p,q}L_{\theta}^s}
		&:=
		\left( p \int_0^\infty t^{q-1} \mu\left\{r>0 \colon
		\n{f(r\,\cdot)}_{L^s(\S^{n-1})}
		\ge t \right\}^{q/p} {dt} \right)^{1/q}
	\end{aligned}
\end{equation}
where $\mu$ is the measure $r^{n-1}dr$ on $\R_+=(0,\infty)$.
The above relations \eqref{eq:schrodinger-enclosures} hold also in the case $\gamma=0$, in the sense that if $D_{0,n} \n{\V}_{L^{n/2}}^{n/2}<1$ for some $D_{0,n}>0$, then the point spectrum of $-\Delta+\V$ is empty (the optimal constant is given by $D_{0,3}=4/(3^{3/2}\pi^2)$ in $3$-dimensions).

The Lieb-Thirring-type bound in \cite{Frank11} are obtained by Frank exploiting two tools: the Birman-Schwinger principle and the Kenig-Ruiz-Sogge estimates in \cite{KenigRuizSogge87} on the conjugate line (see Lemma\til\ref{lem:BS} and Lemma\til\ref{resest-T1} below, respectively). The combination of the Birman-Schwinger principle with resolvent estimates for free operators is one of the way to approach the localization problem for eigenvalues: it has been widely employed in the later times (see e.g. \cite{Frank11,CueninLaptevTretter14,Enblom16,FrankSimon17,Cuenin17,FanelliKrejcirikVega18-JST,FanelliKrejcirik19,CassanoIbrogimovKrejcirikStampach20,CassanoPizzichilloVega20,DAnconaFanelliSchiavone21,DAnconaFanelliKrejcirikSchiavone21} among others) and it will be the approach we are going to follow in this work too, as we will see. However, despite the
robustness of the Birman-Schwinger principle, it is not the only tool one could use to obtain spectral enclosures for non-selfadjoint operators: another powerful technique is the method of multipliers, see e.g. \cite{FanelliKrejcirikVega18-JST,FanelliKrejcirikVega18-JFA,Cossetti17,CossettiFanelliKrejcirik20,CossettiKrejcirik20}.

The Laptev-Safronov conjecture certainly can not be true for $\gamma>n/2$, as observed originally by Laptev and Safronov themselves (see also B\"ogli \cite{Bogli17} for the construction of bounded potentials in $L^{\gamma+n/2}$, $\gamma>n/2$, with infinitely many eigenvalues which accumulate to the real non-negative semi-axis). The situation in the range $1/2<\gamma\le n/2$ is still unclear, even if an argument in \cite{FrankSimon17} suggests that, for these values of $\gamma$, the Laptev-Safronov conjecture should fail in general.
Nevertheless, for $n\ge1$ and $\gamma>1/2$, Frank in \cite{Frank18} proved a localization result still involving the $L^{\gamma+n/2}$ norm of the potential, but in an unbounded region of the complex plane around the semi-line $\sigma(-\Delta)=[0,\infty)$, viz.
\begin{equation}\label{eq:S-dist}
	|z|^{1/2} \dist(z,[0,\infty))^{\gamma-1/2} \le D_{\gamma,n} \n{\V}_{L^{\gamma+n/2}}^{\gamma+n/2}.
\end{equation}
In the limiting case $\gamma=\infty$ one has the trivial bound
\begin{equation}\label{eq:S-dist-oo}
	\dist(z,[0,\infty)) \le D_{\infty,n} \n{\V}_{L^\infty}.
\end{equation}
Thus, it seems that to go beyond the threshold $\gamma=1/2$, one should ask radial symmetry on the potential, or abandon the idea of localizing the eigenvalues in compact regions (cf. Section\til\ref{sec:estimates} below).

To conclude the recap on the spectral results for the Schr\"odinger operator, besides the ones related to the above conjecture, one should refer also to \cite{FrankLaptevSeiringer06}, where bounds on sums of eigenvalues outside a cone
around the positive axis were proved, and to the works \cite{DaviesNath02, Safronov10, Enblom16, FrankSabin17,FanelliKrejcirikVega18-JST, Frank18, LeeSeo19, Cuenin20}, where one can find Keller-type inequalities involving not only the $L^p$ norms.

We now turn our attention to the non-selfadjoint  Dirac operator, formally defined as $\D_{m,\V} = \D_{m}+\V$. The free Dirac operator $\D_{m}$ is given by
\begin{equation*}
	\D_{m} := -i \sum_{k=1}^{n} \alpha_k \partial_k + m \alpha_{n+1}
\end{equation*}
where $n\ge1$ is the dimension, $m\ge0$ is the mass and, set $N:= 2^{\left\lceil n/2 \right\rceil}$ being $\left\lceil\cdot\right\rceil$ the ceiling function, $\alpha_k \in \C^{N\times N}$ are the Dirac matrices, i.e. Hermitian matrices elements of the Clifford algebra, satisfying the anticommutation relations
\begin{equation}\label{clifford}
\alpha_k\alpha_j + \alpha_j\alpha_k = 2\delta_{k}^j I_N
\quad
\text{for $j,k\in\{1,\dots,n\}$}
\end{equation}
where $\delta_k^j$ is the Kronecker symbol and $I_N$ the $N\times N$ unit matrix (we will return on the Dirac matrices later in Section\til\ref{sec:potential}). The potential $\V \colon \R^n \to \C^{N\times N}$ is a possibly non-Hermitian matrix-valued function.

The spectral studies for $\D_{m,\V}$ started with Cuenin, Laptev and Tretter \cite{CueninLaptevTretter14} in the $1$-dimensional case, followed by 
\cite{Cuenin14,CueninSiegl18,Enblom18}. In \cite{CueninLaptevTretter14}, the authors proved that if $\V$ is a $2\times2$ complex matrix with all its entries in $L^1(\R)$, such that
\begin{equation*}
\n{\V}_{L^1(\R)} = \int_{\R} |\V(x)|dx < 1,
\end{equation*}
where $|\V(\cdot)|$ is the operator norm of $\V(\cdot)$ in $\C^2$ with the Euclidean norm,
then every non-embedded eigenvalue $z\in \rho(\D_{m})$ of $\D_{m,\V}$ lies in the union
\begin{equation*}
z \in \overline{B}_{R_0}(x^-_0) \cup \overline{B}_{R_0}(x^+_0)
\end{equation*}
of two disjoint closed disks, with centers and radius respectively
\begin{equation*}
x^\pm_0 = \pm m \sqrt{\frac{\n{\V}_1^4-2\n{\V}_1^2+2}{4(1-\n{\V}_1^2)} + \frac{1}{2}},
\qquad
R_0 = m
\sqrt{\frac{\n{\V}_1^4-2\n{\V}_1^2+2}{4(1-\n{\V}_1^2)} - \frac{1}{2}}.
\end{equation*}
Again, the proof relies on the combination of the Birman-Schwinger principle with a resolvent estimate for the free Dirac operator, namely
\begin{equation*}
\n{(\D_{m}-z)^{-1}}_{L^1(\R) \to L^\infty(\R)}
\le
\sqrt{\frac{1}{2}
	+\frac{1}{4} \left\lvert\frac{z+m}{z-m}\right\rvert
	+\frac{1}{4} \left\lvert\frac{z-m}{z+m}\right\rvert
},
\quad
\text{$z \in \rho(\D_{m})$}.
\end{equation*}
In some sense, this is the counterpart for the Dirac operator of the Abramov-Aslanyan-Davies inequality \eqref{eq:AAD} for the Schr\"odinger operator in $1$-dimension. 

One could ask if, in the same fashion of the Frank's argument in \cite{Frank11}, one can combine the Birman-Schwinger principle with $L^p-L^{p'}$ resolvent estimates for the free Dirac operator, to derive Keller-type inequalities for the perturbed Dirac operator. Unfortunately, these reasoning can not be straightforwardly applied, since such Kenig-Ruiz-Sogge-type estimates does not exists in the case of Dirac for dimensions $n\ge2$, as observed by Cuenin in \cite{Cuenin14}. Indeed, due to the Stein-Thomas restriction theorem and standard estimates for Bessel potentials, the resolvent $(\D_{m}-z)^{-1} \colon L^p(\R^n) \to L^{p'}(\R^n)$ is bounded uniformly for $|z|>1$ if and only if
\begin{equation*}
\frac{2}{n+1}
\le
\frac{1}{p} + \frac{1}{p'}
\le
\frac{1}{n},
\end{equation*}
hence the only possible choice is $(n,p,p')=(1,1,\infty)$. For the Schr\"odinger operator the situation is much better since the right-hand side of the above range is replaced by $2/n$, as per the Kenig-Ruiz-Sogge estimates.

For the high dimensional case $n\ge2$, we refer among others to the works 
\cite{Dubuisson14,CueninTretter16,Cuenin17,FanelliKrejcirik19} where the eigenvalues are localized in terms of $L^p$ norm of the potential, but the confinement region is unbounded around the spectrum $\sigma(\D_{m})=(-\infty,-m] \cup [m,+\infty)$ of the free Dirac operator $\D_{m}$. Instead, we are mainly devoted to the research of a compact region in which to localize the point spectrum. 

Some progress in this direction are achieved by the second author together with D'Ancona and Fanelli in \cite{DAnconaFanelliSchiavone21} and D'Ancona, Fanelli and Krej\v{c}i\v{r}\'{\i}k in \cite{DAnconaFanelliKrejcirikSchiavone21}. In the first work, we proved a result that generalizes in higher dimensions the previous one by Cuenin, Laptev and Tretter \cite{CueninLaptevTretter14}. Indeed, assuming $\V$ small enough respect to a suitable mixed Lebesgue norm, namely
\begin{equation}\label{eq:DFS}
	\n{\V}_Y :=
	\max_{j\in\{1,\dots,n\}} \n{ \V }_{L^1_{x_j} L^\infty_{{\widehat{x}_j}}}
	=
	\max_{j\in\{1,\dots,n\}} 
	\int_{\R} \n{\V(x_j,\cdot)}_{L^\infty(\R^{n-1})} dx_j
	\le C_0
\end{equation}
for a positive constant $C_0$ independent of $\V$, we prove in the massive case $m>0$ that the eigenvalues of $\D_{m,\V}$ are contained in the union 
	\begin{equation*}
	z \in \overline{B}_{R_1} (x^-_1) \cup \overline{B}_{R_1} (x^+_1)
	\end{equation*}
	of the two closed disks in $\C$ with centers and radius given by
	\begin{equation}\label{eq:disks-DFS}
	x^\pm_1 := \pm m\,\frac{\nu^2+1}{\nu^2-1},
	\quad
	R_1 := m\,\frac{2\nu}{\nu^2-1},
	\quad\text{where}\quad 
	\nu := \left[\frac{(n+1)C_0}{\n{\V}_Y} - n \right]^2 > 1.
	\end{equation}
Instead, in the massless case $m=0$, the spectrum is stable respect to the perturbation $\V$, viz.	$\sigma(\D_{0,\V})=\sigma_{e}(\D_{0,\V})=\R$ and there are no eigenvalues, under the same smallness assumption for the potential. This results are proved combining the Birman-Schwinger principle together with a new Agmond-H\"ormander-type estimates for the resolvent of the Schr\"odinger operator and its first derivatives.

The same machinery is employed in \cite{DAnconaFanelliKrejcirikSchiavone21}, where again we take advantage of the main engine of the Birman-Schwinger operator fueled this time with resolvent estimates already published in the literature, but which imply spectral results for the Dirac operator (and for the Klein-Gordon one) worthy of consideration. In particular, in dimension $n\ge3$ we show again results similar to the previous ones, hence confinement of the eigenvalues in two disks in the massive case and their absence in the massless case, assuming now for the potential the smallness assumption 
\begin{equation*}
	\n{|x|\V}_{\ell^1L^\infty} := \sum_{j\in\Z} \n{|x|\V}_{L^\infty(2^{j-1}\le|x|<2^j)} < C_1 
\end{equation*}
and substituting the definition of $\nu$ in \eqref{eq:disks-DFS} with $\nu:=[2C_1/\n{|x|\V}_{\ell^1L^\infty}-1]^2$. The constant $C_1$ can be explicitly showed as a number depending only on the dimension $n$ and, even if far to be optimal, is still valuable in the application. Moreover, in \cite{DAnconaFanelliKrejcirikSchiavone21} results for the stability of the spectrum are proved not only in the massless case, but also in the massive one, assuming smallness pointwise assumptions on the weighted potential, namely $\n{|x|\rho^{-2} \V}_{L^\infty}<C_2$. The constant $C_2$ is made explicit in terms of the dimension $n$ and the mass $m$, and $\rho$ is a positive weight satisfying $\sum_{j\in\Z} \n{\rho}^2_{L^\infty(2^{j-1}\le|x|<2^j)}<\infty$ and additionally, in the massive case, such that $|x|^{1/2}\rho \in L^\infty(\R^n)$ (prototypes of such kind of weights already appeared e.g. in \cite{BarceloRuizVega97}).

Looking at the above mentioned results of spectral enclosure for non-selfadjoint Dirac operators, two situations seems to arise: or the confinement regions are unbounded, containing the continuous spectrum of the free Dirac operator $\D_{m}$, or the regions are bounded, but the potential is required to be small respect to some \lq\lq cumbersome'' norm. 

We finally mention the works \cite{ErdoganGoldbergGreen19} by Erdo\u{g}an, Goldberg and Green, and \cite{ErdoganGreen21} by Erdo\u{g}an and Green, 
where the authors, studying the limiting absorption principle and dispersive bounds, prove that for a bounded, continuous potential $V$ satisfying a mild decaying condition, there are no eigenvalues of the perturbed Dirac operator in a sector of the complex plane containing a portion of the real line sufficiently far from zero energy. These results are qualitative, in the sense that their bounds does not explicitly depend on some norm of the potential, as in the inequalities object of our study.

In the present paper we recover Keller-type bounds which are a worthy analogous of the Schr\"odinger enclosures in \eqref{eq:schrodinger-enclosures}, hence exploiting $L^p$ norms at least for $n/2 \le p \le (n+1)/2$; also, we can remove the smallness assumption on the potential (when $p\neq n/2$). Of course, to reach such a nice result, the price to pay is high: we will require to our potentials to be of the form $\V=vV$, where $v:\R^n\to\C$ is a scalar function in the desired space of integrability, whereas $V$ is a constant matrix satisfying some suitable rigidity conditions.
%
Hence, in a way to be clarified later, we will fully take advantage of the matricial structure of the Dirac operator in order to reduce ourselves basically to the Schr\"odinger case.

The paper is structured as follows: in the next Section\til\ref{sec:main} we will state our main results, proved in Section\til\ref{sec:BS} employing the Birman-Schwinger principle there recalled and the resolvent estimates for the Schr\"odinger operator collected in Section\til\ref{sec:estimates}. Finally, in Section\til\ref{sec:potential}, we explicitly compute examples of potentials satisfying our rigidity assumptions.


\section{Idea and main results}\label{sec:main}

As anticipated in the Introduction, the trick of our argument relies completely on the matricial structure of the potential $\V$. Before to rattle off the hypothesis we are going to impose on it, let us recall the basic idea behind the Birman-Schwinger principle. In order to make things work and being formal, just for the moment assume that $\V$ is bounded, such that $\D_{m,\V}=\D_{m}+\V$ is well defined as sum of operators. We will return in great generality on this matter later, in Section\til\ref{sec:BS}, where we will also be able to properly define $\D_{m,\V}$ thanks to Lemma\til\ref{lem:H_V}.

The principle assure us that $z$ is an eigenvalue of $\D_{m,\V}$, where $\V=\B^*\A$ is a factorizable potential, if and only if $-1$ is an eigenvalue of the Birman-Schwinger operator $K_z:=\A(\D_{m}-z)^{-1}\B^*$ (the computation can be straightforwardly checked). If $-1 \in \sigma_p(K_z)$ then $\n{K_z}\ge1$, which turns out to be the desired localization bound, if one is able to estimate the Birman-Schwinger operator.

From the well-known identity
\begin{equation}\label{eq:dirac-resolvent-identity}
	(\D_{m}-z)^{-1} = (\D_{m}+z) R_0(z^2-m^2) I_{N}
\end{equation}
which links the Dirac resolvent $(\D_{m}-z)^{-1}$ with the Schr\"odinger resolvent $R_0(z):=(-\Delta-z)^{-1}$, we have that 
\begin{equation}\label{Kz-1}
	\A (\D_{m}-z)^{-1} \B^* = -i \sum_{k=1}^{n} \A \alpha_k \partial_k R_0(z^2-m^2) \B^* + \A (m \alpha_{n+1} + z) R_0(z^2-m^2) \B^*.
\end{equation}
At this point, the receipt one usually cooks (as in the literature works cited in the Introduction) is the following. First of all, the polar decomposition $\V = \mathcal{U}\mathcal{W}$ of the potential is exhibited, where $\mathcal{W}=\sqrt{\V^* \V}$ and the unitary matrix $\mathcal{U}$ is a partial isometry. Then one takes $\A=\sqrt{\mathcal{W}}$ and $\B=\sqrt{\mathcal{W}}\mathcal{U}^*$; this choice assures a certain symmetry in splitting the potential, since $\A$ and $\B$ are in the same space of integrability. Therefore, making use of resolvent estimates and of the H\"older's inequality, one reaches an estimate of the form $1\le \n{K_z} \le \kappa(z) \n{\V}_X$ for some suitable function $\kappa:\C\to\R$ and space $X$.

Clearly, the main problem is reduced to the research of nice resolvent estimates. For the Schr\"odinger operator, these have been extensively studied, so if we look at \eqref{Kz-1} the main concern comes from the estimates for the derivatives of $R_0(z)$. Our idea here is to choose $\A$ and $\B$ in such a way that the terms $\A \alpha_k \partial_k R_0(z^2-m^2) \B^*$, for any $k\in\{1,\dots,n\}$, simply disappear 
(we will make an exception to this for Theorem\til\ref{thm:RAi-general}). 
If additionally we impose also $\A R_0(z^2-m^2) \B^*$ to be zero, we are also able to remove the smallness assumption on the potential, because it turns out that they originates from this term.
Therefore, let us state the following hypothesis.

\begin{rigidityassumptions*}
Let us consider potential of the type $\V=vV=\B^*\A$, $\A=aA$, $\B=bB$, in such a way that $v=\overline{b}a$ and $V=B^*A$, where $a,b,v \colon \R^n \to \C$ are complex-valued functions and $A,B,V \in \C^{N\times N}$ are constant matrices. 

On the scalar part $v$, we impose the usual polar decomposition, viz. $a=|v|^{1/2}$ and $b=\overline{\sgn(v)}|v|^{1/2}$, where the sign function is defined as $\sgn(w)=w/|w|$ for $0\neq w \in \C$ and $\sgn(0)=0$.

On the matricial part $V$, we ask the following set of conditions:
	\begin{equation*}\label{eq:AB}
		\begin{gathered}
			A \alpha_k B^*=0
			\quad\text{for $k \in\{1,\dots,n\}$,}
			\\
			V=B^*A \neq 0.
		\end{gathered}
	\end{equation*}
It is not restrictive to assume also that 
\begin{equation*}\label{eq:AB-normalization}
|A|=|B|=1
\end{equation*}
where $|\cdot|\colon\C^{N\times N}\to \R$ is the norm induced by the Euclidean norm, viz. $|A|=\sqrt{\rho(A^*A)}$, where $\rho(M)$ is the spectral radius of $M$.

In addition to the above stated hypothesis, suppose also one between the next conditions:
\begin{enumerate}[label=(\roman*)]
	\item $A\alpha_{n+1} B^* \neq 0$ and $AB^* \neq 0$;
	\item $A\alpha_{n+1} B^* \neq 0$ and $AB^* = 0$;
	\item $A\alpha_{n+1} B^* = 0$ and $AB^* \neq 0$;
	\item $A\alpha_{n+1} B^* = 0$ and $AB^* = 0$.
\end{enumerate}
\end{rigidityassumptions*}

In the following, we will refer to our set of rigidity assumptions as RA($\iota$), where $\iota\in\{i,ii,iii,iv\}$ depends on which of the four conditions above is considered.

\begin{remark}
	Note that we will \textit{not} assume any Rigidity Assumptions in Theorem\til\ref{thm:RAi-general}, but only in Theorems\til\ref{thm:RAii-n1}--\ref{thm:RAi-dist} below.
\end{remark}

\begin{remark}\label{rem:dirac}
	At this point the reader may argue that the assumptions above are not rigorous, since we have not explicitly defined the Dirac matrices $\alpha_k$, $k\in\{1,\dots,n+1\}$. Moreover, there is not a unique representation for these matrices! The concern is legit, and we will furnish later the exact definitions of our Dirac matrices, in Section\til\ref{sec:potential}, which will be all devoted to computations with matrices. The choice of a particular representation of the Dirac matrices is not restrictive, see Remark\til\ref{rem:not-restrictive}.
\end{remark}

\begin{remark}
	As will be proved in Section\til\ref{sec:potential}, we can find matrices $A$ and $B$ satisfying RA(i) in any dimension $n\ge1$, whereas there are no matrices satisfying RA(ii) and RA(iii) in dimensions $n=2,4$ and no matrices satisfying RA(iv) in dimensions $n=1,2$. This explains the dimensions restriction in the statements of the theorems below.
\end{remark}

We can state now our main results. 
Recall, other than the Lebesgue norm, the Lorentz norm and the radial-angular norm introduced in \eqref{eq:radial-angular}.
We refer to Figures\til\ref{fig:thm123},\til\ref{fig:thm56} and\til\ref{fig:thm89} to visualize the boundary curves of the confinement regions described in the various theorems.


Let us start considering the case of RA(ii).

\begin{theorem}\label{thm:RAii-n1}
	Let $m>0$, $n=1$ and $\V=vB^*A$ satisfying RA(ii). Then 
	\begin{equation*}
	|z^2-m^2|^{1/2} \le \frac{1}{2} \n{v}_{L^1}
	\end{equation*}
	for any $z\in\sigma_p(\D_{m,\V})$.
\end{theorem}

\begin{theorem}\label{thm:RAii}
	Let $m>0$, $n\in\N\setminus\{1,2,4\}$ and $\V=vB^*A$ satisfying RA(ii). There exists $D_{\gamma,n,m}>0$
	such that
	\begin{equation*}
	|z^2-m^2|^{\gamma} \le D_{\gamma,n,m}
	\left\{
	\begin{aligned}
	&\n{v}_{L^{\gamma+n/2}}^{\gamma+n/2}
	&&\text{for $0<\gamma\le\frac{1}{2}$,}
	\\
	&\n{v}_{L_\rho^{\gamma+n/2} L_\theta^\infty}^{\gamma+n/2}
	&&\text{for $\frac{1}{2}<\gamma< \frac{n}{2}$,}
	\\
	&\n{v}_{L_\rho^{n,1} L_\theta^\infty}^{n}
	&&\text{for $\gamma=\frac{n}{2}$,}
	\end{aligned}
	\right.
	\end{equation*}
	for any $z\in\sigma_p(\D_{m,\V})$.
	
	In the case $\gamma=0$, there exists $D_{0,n}>0$ such that, if
	\begin{equation*}
		\n{v}_{L^{n/2}} < D_{0,n,m}
	\end{equation*}
	then $$\sigma(\D_{m,\V})=\sigma_c(\D_{m,\V})=\sigma(\D_{m})=(-\infty,-m]\cup[m,\infty)$$ and in particular $\sigma_p(\D_{m,\V})=\varnothing$.
\end{theorem}

\begin{theorem}\label{thm:RAii-dist}
	Let $m>0$, $n\in\N\setminus\{1,2,4\}$, $\gamma>1/2$ and $\V=vB^*A$ satisfying RA(ii). There exists $D_{\gamma,n,m}>0$ such that
	\begin{equation*}
	|z^2-m^2|^{1/2}  
	\dist(z^2-m^2,[0,\infty))^{\gamma-1/2} 
	\le D_{\gamma,n,m} \n{v}_{L^{\gamma+n/2}}^{\gamma+n/2}
	%
	%
	\end{equation*}
	for any $z\in\sigma_p(\D_{m,\V})$. In the case $\gamma=\infty$, the above relation is replaced by
	\begin{equation*}
	\dist(z^2-m^2,[0,\infty)) \le D_{\infty,n,m} \n{v}_{L^\infty}.
	\end{equation*}
\end{theorem}

\begin{remark}
	Note that, since
	\begin{equation*}
	\dist(z, [0,\infty))
	=
	\left\{
	\begin{aligned}
	& |\Im z| &&\text{if $\Re z \ge 0$,}
	\\
	& |z| &&\text{if $\Re z \le 0$,}
	\end{aligned}
	\right.
	\end{equation*}
	then
	\begin{equation*}
	\dist(z^2-m^2,[0,\infty))
	=
	\left\{
	\begin{aligned}
	& 2|\Re z||\Im z| &&\text{if $(\Re z)^2-(\Im z)^2 \ge m^2$,}
	\\
	& |z^2-m^2| &&\text{if $(\Re z)^2-(\Im z)^2 \le m^2$.}
	\end{aligned}
	\right.
	\end{equation*}
\end{remark}


\begin{figure}[p!]
	\begin{subfigure}[b]{\textwidth}
		\centering
		\resizebox{\linewidth}{!}{
			\includegraphics[scale=1]{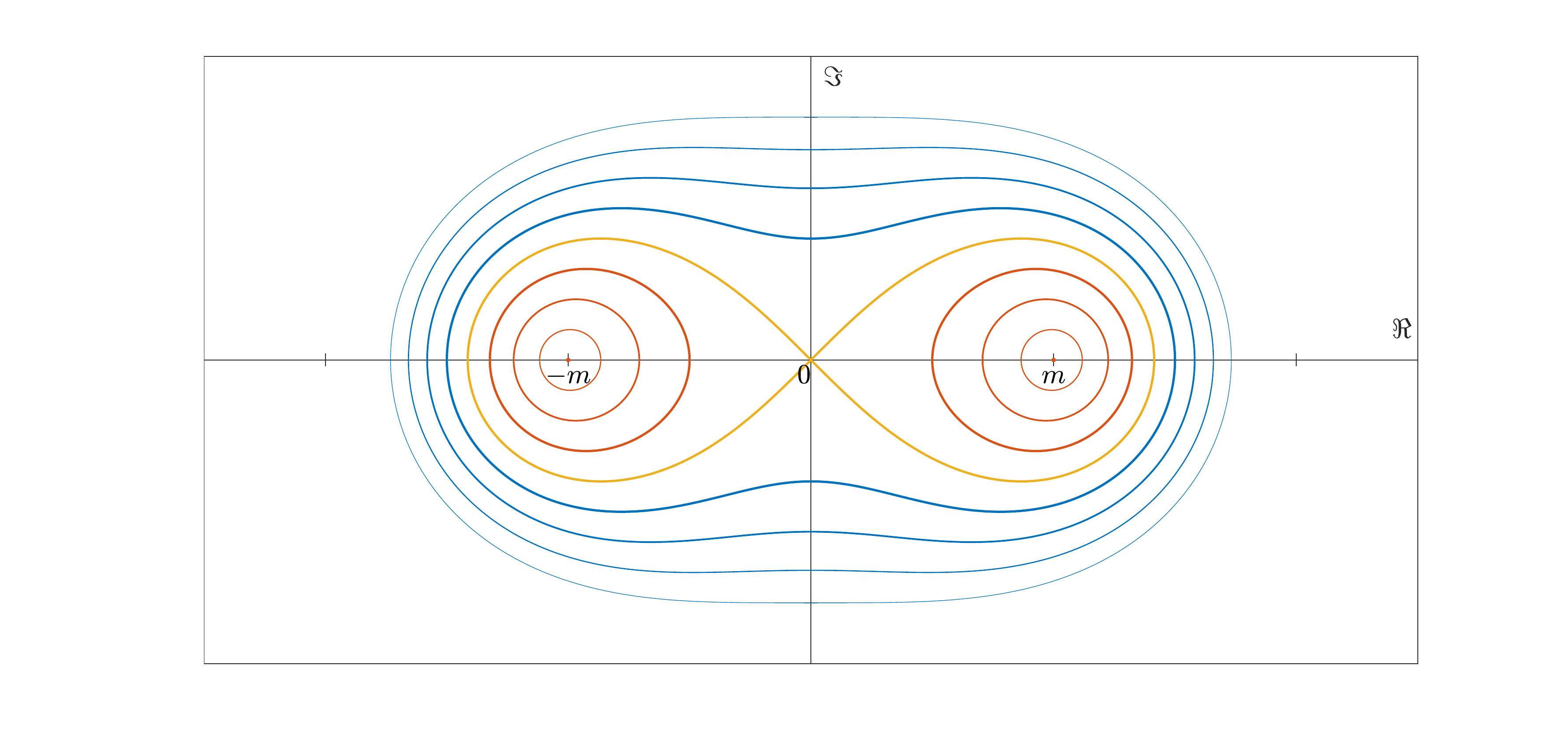}
		}
		\caption{Case of Theorems\til\ref{thm:RAii-n1} and\til\ref{thm:RAii}.}
	\end{subfigure}
	
	\begin{subfigure}[b]{\textwidth}
		\centering
		\resizebox{\linewidth}{!}{
			\includegraphics[scale=1]{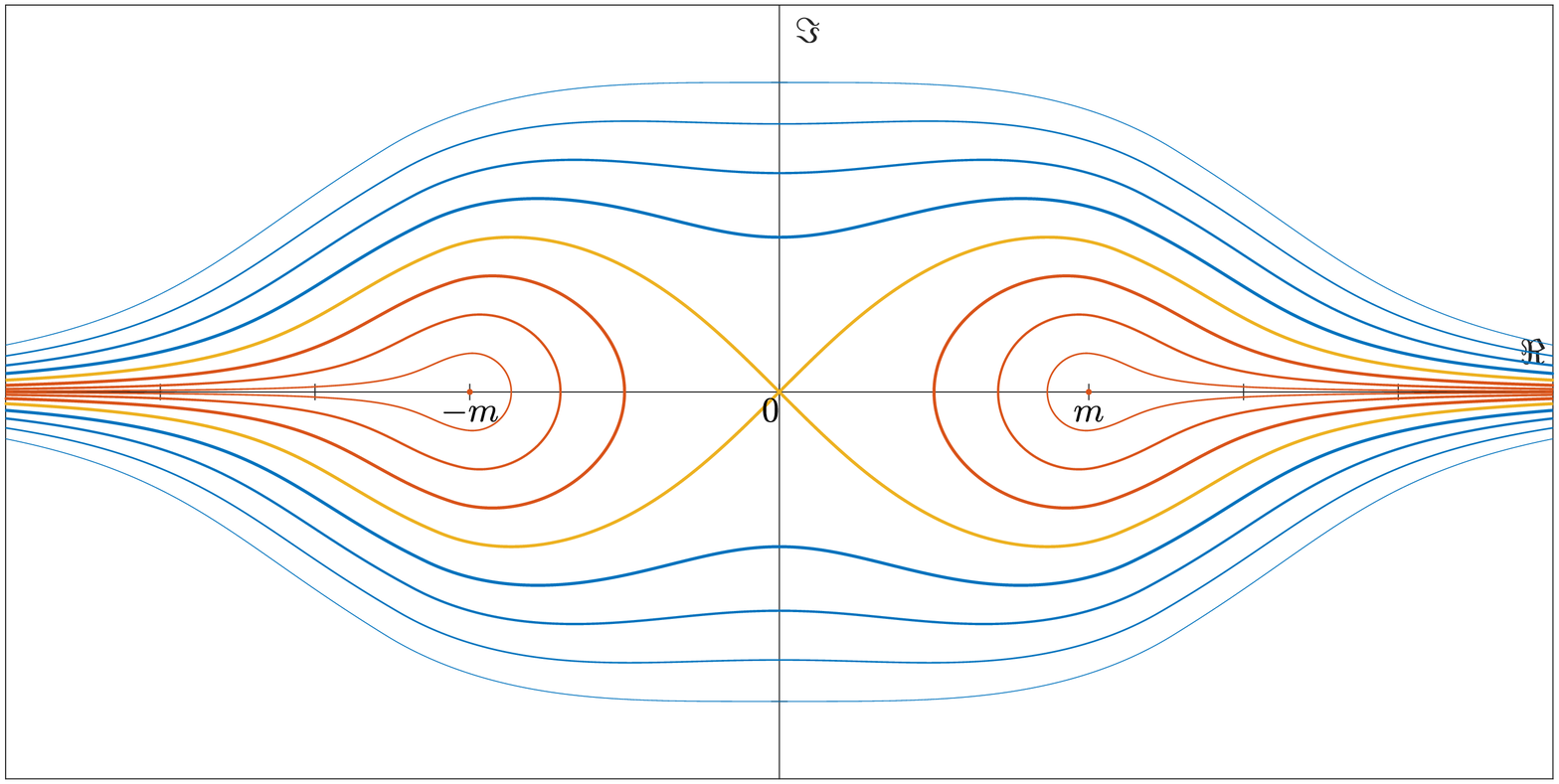}
		}
		\caption{Case of Theorem\til\ref{thm:RAii-dist}.}
	\end{subfigure}
	\caption{\small The plots of the boundary curves corresponding to the spectral enclosures described in Theorems\til\ref{thm:RAii-n1},\til\ref{thm:RAii} and\til\ref{thm:RAii-dist}, for various values of the norm of the potential. 
	\newline
	When $\beta := D_{\gamma,n,m}\n{v}^{\gamma+n/2}=1$, where $D_{1/2,1,m}=1/2$ and $\n{v}$ is one of the norms appearing in the theorems, we have two regions joined only in the origin (in yellow). If $\beta<1$ there are two disconnected regions (in red), while if $\beta>1$ there is one connected region (in blue). 
	\newline
	The curves in picture (a) are known as Cassini ovals with foci in $m$ and $-m$.}
	\label{fig:thm123}
\end{figure}


The results collected in the three theorems above should be compared with the corresponding ones for the Schr\"odinger operator, respectively \eqref{eq:AAD}, \eqref{eq:schrodinger-enclosures} and \eqref{eq:S-dist}--\eqref{eq:S-dist-oo}.
We supposed RA(ii) with positive mass $m>0$, which means, looking \eqref{Kz-1}, that 
\begin{equation*}
	\A (\D_m -z)^{-1} \B^*
	=
	m [A \alpha_{n+1} B^*] [ aR_0(z^2-m^2) \overline{b} ].
\end{equation*}
Roughly speaking, the Birman-Schwinger operator for $\D_{m}+\V$ behaves (more or less) as the Birman-Schwinger operator for $-\Delta+v$. This explains the strict connection between the Dirac and Schr\"odinger results.

If we consider RA(ii) with $m=0$, or instead RA(iv), then the Birman-Schwinger operator for Dirac vanish identically, implying the following result of spectral stability.

\begin{theorem}\label{thm:RAii-iv}
	Let $n\in\N\setminus\{2,4\}$, $m=0$ and $\V=vB^*A$ satisfying RA(ii), or alternatively $n\in\N\setminus\{1,2\}$, $m\ge0$ and $\V=vB^*A$ satisfying RA(iv).
	Then $$\sigma(\D_{m,\V})=\sigma_c(\D_{m,\V})=\sigma(\D_{m})=(-\infty,-m]\cup[m,\infty)$$ and in particular $\sigma_p(\D_{m,\V})=\varnothing$.
\end{theorem}

We stress out again that the above results does not require any smallness assumption on the potential, even if, of course, the regions of confinement described in Theorems\til\ref{thm:RAii-n1},\til\ref{thm:RAii} and\til\ref{thm:RAii-dist} become larger and larger when the norm of $v$ increases.

Let us wonder now what happens removing the condition $AB^*=0$. As we see from the following theorems, the requirement that the potential should be small pops up again. Moreover, we find a compact localization for the eigenvalues (or their absence) only respect to the $L^1$-norm when $n=1$, and to the $L_\rho^{n,1} L_\theta^\infty$-norm when $n\ge2$. 

About the localization around the continuous spectrum of the free operator, it is not so nice as that in Theorem\til\ref{thm:RAii-dist}, where the region of confinement, even if unbounded, \lq\lq narrows'' around $\sigma(\D_{m})$. Denoting for simplicity with $\mathcal{N}$ one of the region described in Theorems\til\ref{thm:RAiii-dist} and\til\ref{thm:RAi-dist}, we have that it \lq\lq become wider'' around $\sigma(\D_{m})$, even if the sections $\mathcal{N} \cap \{z\in\C \colon \Re z = x_0\}$ are compact for any fixed $x_0\in\R$. Also, we need to require $\gamma\ge n/2$, otherwise the region $\mathcal{N}$ would be the complement of a bounded set, and hence not so interesting (see Subsection\til\ref{subsec:proof}).

Hence, let us state now the results assuming RA(iii) and RA(i) respectively.


\begin{theorem}\label{thm:RAiii}
	Let $n\in\N\setminus\{2,4\}$, $m\ge0$ and $\V=vB^*A$ satisfying RA(iii). 
	Moreover, let us set for simplicity 
	\begin{equation*}
		\n{\cdot} := 
		\begin{cases}
			\n{\cdot}_{L^1} & \text{if $n=1$,}
			\\
			\n{\cdot}_{L^{n,1}_\rho L^\infty_\theta} & \text{if $n\ge2$.}
		\end{cases}
	\end{equation*}
	%
	%
	There exists $C_0>0$ such that, if $\n{v}<C_0$ and $m>0$, then
	\begin{equation*}
		|z^2-m^2|^{1/2} |z|^{-1} 
		\le 
		C_0^{-1} \n{v}
	\end{equation*}
	for any $z\in \sigma_p(\D_{m,\V})$, whereas, if $\n{v}<C_0$ and $m=0$, then
	\begin{equation*}
		\sigma(\D_{0,\V})=\sigma_c(\D_{0,\V})=\sigma(\D_{0})=\R
	\end{equation*}
	and in particular $\sigma_p(\D_{0,\V})=\varnothing$.
	
	If $n=1$, we can take $C_0=2$.
\end{theorem}


\begin{figure}[p!]
	
	\begin{subfigure}[b]{\textwidth}
		\centering
		\resizebox{\linewidth}{!}{
			\includegraphics[scale=1]{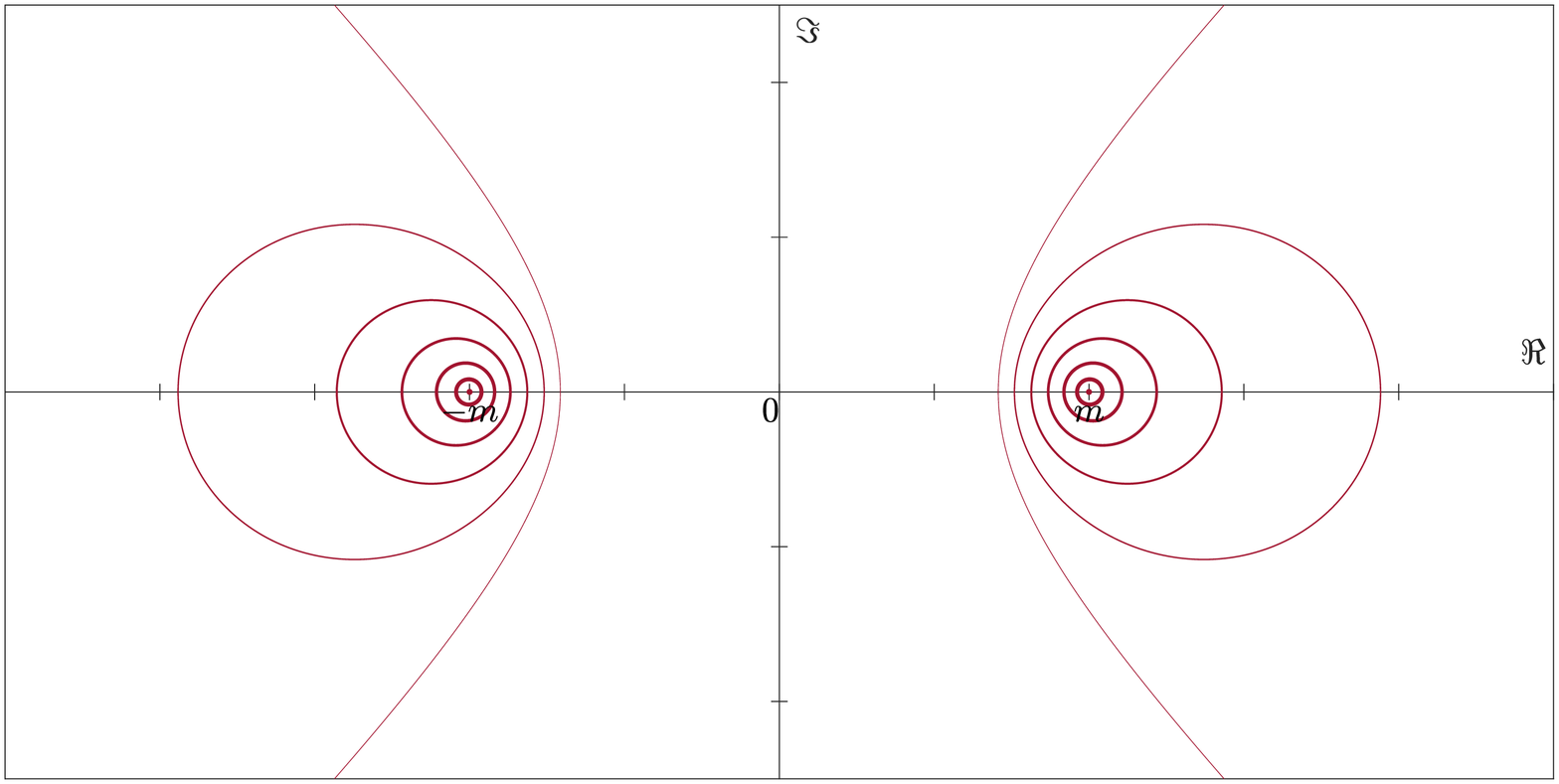}
		}
		\caption{Case of Theorem\til\ref{thm:RAiii}.}
	\end{subfigure}
	
	\begin{subfigure}[b]{\textwidth}
		\centering
		\resizebox{\linewidth}{!}{
			\includegraphics[scale=1]{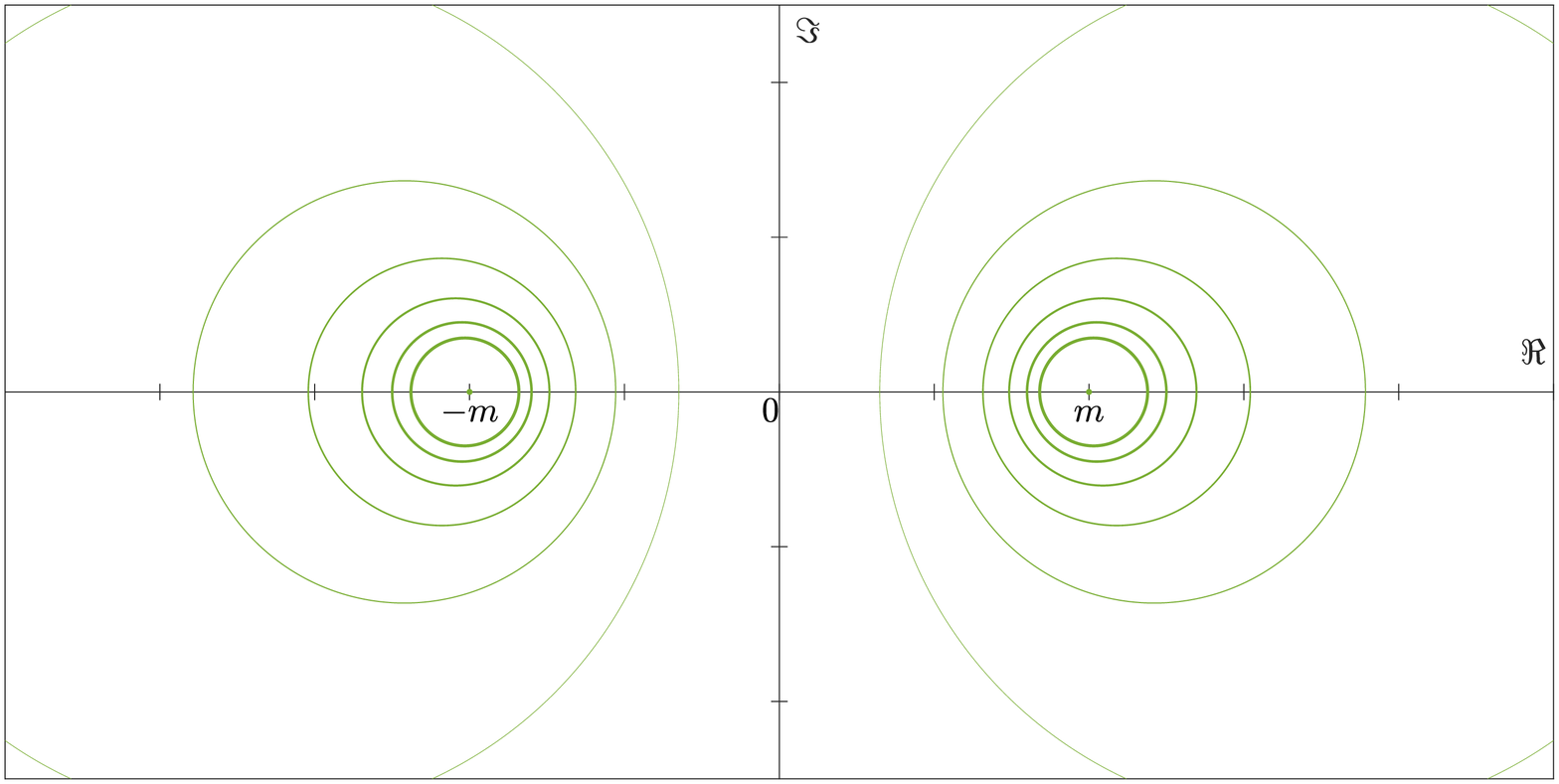}
		}
		\caption{Case of Theorems\til\ref{thm:RAi} and\til\ref{thm:RAi-general}.}
	\end{subfigure}

	\caption{\small The plots of the boundary curves corresponding to the spectral enclosures described in Theorem\til\ref{thm:RAiii} and Theorems\til\ref{thm:RAi} and\til\ref{thm:RAi-general}, for various values of the norm of the potential. The region is always the union of two disconnected components.}
	\label{fig:thm56}
\end{figure}


\begin{theorem}\label{thm:RAiii-dist}
	Let $n\in\N\setminus\{1,2,4\}$, $m\ge0$, $\V=vB^*A$ satisfying RA(iii) and $\gamma \ge n/2$. Then there exists $C_0>0$ such that 
	\begin{equation*}
		|z^2-m^2|^{1/2}
		|z|^{-\gamma-n/2}
		\dist(z^2-m^2,[0,\infty))^{\gamma-\frac{1}{2}}
		\le
		C_0^{-1}
		\n{v}_{L^{\gamma+n/2}}^{\gamma+n/2}
	\end{equation*}
	for any $z \in \sigma_p(\D_{m,\V})$.
	If $\gamma=\infty$, the above relation is substituted by
	\begin{equation*}
		|z|^{-1}
		\dist(z^2-m^2,[0,\infty))
		\le
		C_0^{-1}
		\n{v}_{L^{\infty}}.
	\end{equation*}
	If $\gamma=n/2$, we should ask also that $\n{v}_{L^{\gamma+n/2}}^{\gamma+n/2} < C_0$.
\end{theorem}


\begin{figure}[p!]
	
	\begin{subfigure}[b]{\textwidth}
		\centering
		\resizebox{\linewidth}{!}{
			\includegraphics[scale=1]{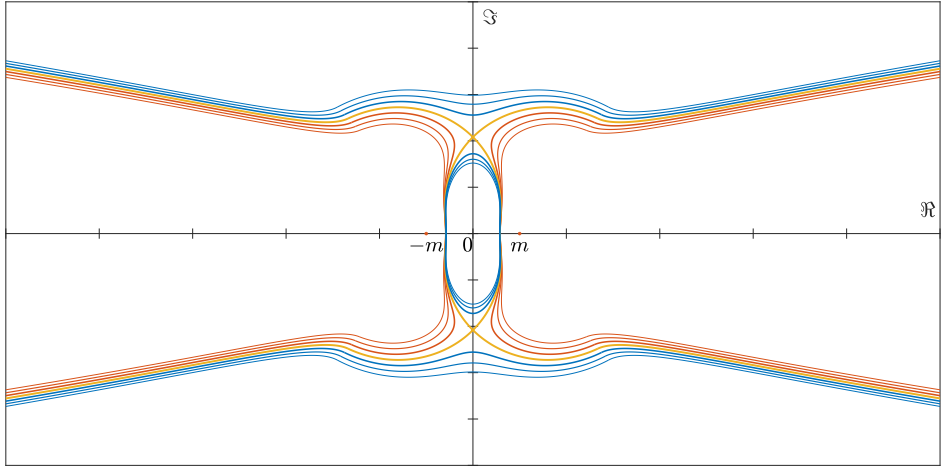}
		}
		\caption{Case of Theorem\til\ref{thm:RAiii-dist}.}
	\end{subfigure}
	
	\begin{subfigure}[b]{\textwidth}
		\centering
		\resizebox{\linewidth}{!}{
			\includegraphics[scale=1]{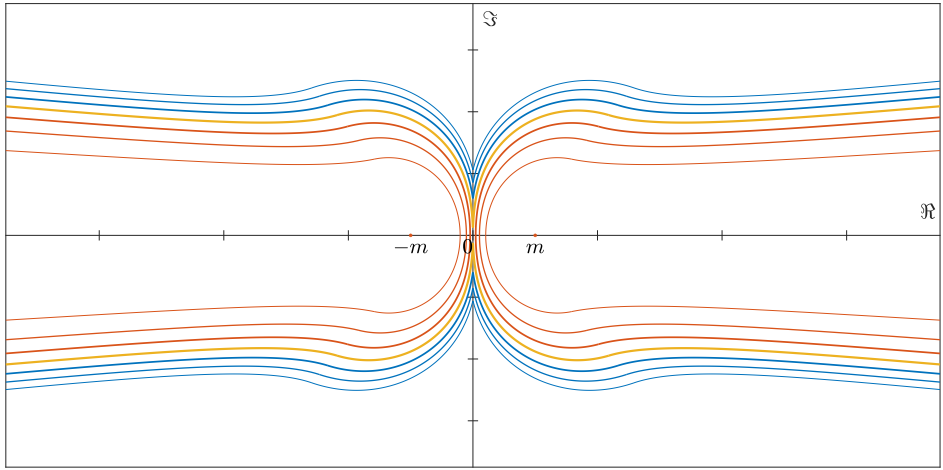}
		}
		\caption{Case of Theorem\til\ref{thm:RAi-dist}.}
	\end{subfigure}

	\caption{\small The plots of the boundary curves corresponding to the spectral enclosures described in Theorems\til\ref{thm:RAiii-dist} and\til\ref{thm:RAi-dist}, for various values of the norm of the potential and for $n/2<\gamma<\infty$. 
		\newline
		According to the value of the norm of $v$, the enclosure region can be composed: by two disconnected components (in red); by two components joining in two points in the case of Theorem\til\ref{thm:RAiii-dist}, and in the origin in the case of Theorem\til\ref{thm:RAi-dist} (in yellow); by one connected region (in blue), which presents a \lq\lq hole'' around the origin in the case of Theorem\til\ref{thm:RAiii-dist}.}
	\label{fig:thm89}
\end{figure}


\begin{theorem}\label{thm:RAi}
	Let $n\ge1$, $m\ge0$ and $\V = vB^*A$ satisfying RA(i).
	Moreover, let us set for simplicity 
	\begin{equation*}
		\n{\cdot} := 
		\begin{cases}
			\n{\cdot}_{L^1} & \text{if $n=1$,}
			\\
			\n{\cdot}_{L^{n,1}_\rho L^\infty_\theta} & \text{if $n\ge2$.}
		\end{cases}
	\end{equation*}
	%
	
	There exists a constant $C_0>0$ such that, if $m>0$ and $\n{v} < C_0$, then the point spectrum of $\D_{m,\V}$ is confined in the union of the two closed disks
	\begin{equation*}\label{eq:circles}
	\sigma_p(\D_\V) \subseteq \overline{B}_{R}(c_+) \cup \overline{B}_{R}(c_-)
	\end{equation*}
	with centers and radius given by
	\begin{equation*}
	c_\pm = \pm m
	\frac{C_0^4+\n{v}^4}{C_0^4-\n{v}^4},
	\qquad
	R=m \frac{2C_0^2\n{v}^2}{C_0^4-\n{v}^4}.
	\end{equation*}

	Instead, if $m=0$ and $\n{v} < C_0$, then
	\begin{equation*}
		\sigma(\D_{0,\V})=\sigma_c(\D_{0,\V})=\sigma(\D_{0})=\R
	\end{equation*}
	and in particular $\sigma_p(\D_{0,\V})=\varnothing$.
	
	If $n=1$, we can take $C_0=2$.
\end{theorem}

\begin{theorem}\label{thm:RAi-dist}
	Let $n\ge2$, $m \ge 0$, $\V = vB^*A$ satisfying RA(i) and $\gamma\ge n/2$. Then there exist $C_0>0$ such that
	\begin{equation*}
		|z^2-m^2|^{\frac{1}{2}\left(1-\gamma-\frac{n}{2}\right)}
		\left\lvert \frac{z+m}{z-m} \right\rvert^{-\left(\gamma+\frac{n}{2}\right)\frac{\sgn\Re z}{2}}
		\dist(z^2-m^2,[0,\infty))^{\gamma-\frac{1}{2}}
		\le
		C_0^{-1}
		\n{v}_{L^{\gamma+n/2}}^{\gamma+n/2}
	\end{equation*}
	for any $z \in \sigma_p(\D_{m,\V})$. If $\gamma=\infty$, the above relation is substituted by
	\begin{equation*}
		|z-m|^{\frac{\sgn\Re z -1}{2}}
		|z+m|^{-\frac{\sgn\Re z +1}{2}}
		\dist(z^2-m^2,[0,\infty))
		\le
		C_0^{-1}
		\n{v}_{L^{\infty}}.
	\end{equation*}
	If $\gamma=n/2$, we should ask also that $\n{v}_{L^{\gamma+n/2}}^{\gamma+n/2} < C_0$.
\end{theorem}


As we already explained, the main trick to get the theorems above basically consists of imposing all the term of the type $\A \alpha_k \partial_k R_0(z^2-m^2) \B^*$ in \eqref{Kz-1} to vanish, leaving only the last term:
\begin{equation*}
	\A(\D_m -z)^{-1} \B^* = 
	A(m\alpha_{n+1}+z) B^*
	\left[aR_0(z^2-m^2) \overline{b}\right].
\end{equation*}
This because we want to employ estimates for the resolvent of the Schr\"odinger operator but not for its derivatives. However, the work \cite{BarceloRuizVega97} furnish us some kind of such estimates for the derivatives of the Schr\"odinger resolvent (see Lemma\til\ref{resest-T2} below). Consequently, we can easily obtain the following confinement result without requiring any special structure on the potential $\V$, but only assuming its smallness respect to the $L_\rho^{n,1} L_\theta^\infty$-norm.

\begin{theorem}\label{thm:RAi-general}
	Let $n\ge2$, $m\ge0$ and $\V \colon \R^n \to \C^{N\times N}$ a generic function.
	%
	%
	There exists a constant $C_0>0$ such that, if $m>0$ and $\n{\V}_{L^{n,1}_\rho L^\infty_\theta} < C_0$, then the point spectrum of $\D_{m,\V}$ is confined in the union of the two closed disks
	\begin{equation*}\label{eq:circles}
		\sigma_p(\D_\V) \subseteq \overline{B}_{R}(c_+) \cup \overline{B}_{R}(c_-)
	\end{equation*}
	with centers and radius given by
	\begin{equation*}
		c_\pm = \pm m
		\frac{\nu^2 + 1}{\nu^2 - 1},
		\qquad
		R=m \frac{2\nu}{\nu^2 - 1},
		\qquad
		\nu := 
		\left[ \frac{2C_0}{\n{\V}_{L^{n,1}_\rho L^\infty_\theta}} - 1 \right]^2.
	\end{equation*}
	
	Instead, if $m=0$ and $\n{\V}_{L^{n,1}_\rho L^\infty_\theta} < C_0$, then
	\begin{equation*}
		\sigma(\D_{0,\V})=\sigma_c(\D_{0,\V})=\sigma(\D_{0})=\R
	\end{equation*}
	and in particular $\sigma_p(\D_{0,\V})=\varnothing$.
\end{theorem}

	The above theorem is a generalization of Theorem\til\ref{thm:RAi}, dropping the many restrictions on $\V$ and with slightly modified definitions of centers and radius of the disks. In some sense, it can be seen as the radial version of the result in \cite{DAnconaFanelliSchiavone21} recalled above in \eqref{eq:DFS}--\eqref{eq:disks-DFS}.

We will prove our results after recalling the key resolvent estimates for the Schr\"odinger operator and the Birman-Schwinger principle.


\section{Resolvent estimates for Schr\"odinger}\label{sec:estimates}

In this section we collect some well-known resolvent estimates for the Schr\"odinger operator.
For our purposes the estimates on the conjugate line 
are sufficient, but we think it is nice to look at the complete picture.

For dimension $n\ge3$, let us define the following endpoints
\begin{gather*}
A := \left(\frac{n+1}{2n} , \frac{n-3}{2n}\right),
\qquad
A' := \left(\frac{n+3}{2n} , \frac{n-1}{2n}\right),
\\
B := \left(\frac{n+1}{2n} , \frac{(n-1)^2}{2n(n+1)}\right) ,
\qquad
B' := \left(\frac{n^2+4n-1}{2n(n+1)} , \frac{n-1}{2n}\right),
\\
A_0 := \frac{A+A'}{2} = \left( \frac{n+2}{2n}, \frac{n-2}{2n} \right),
\qquad
B_0 := \frac{B+B'}{2} = \left( \frac{n+3}{2(n+1)} , \frac{n-1}{2(n+1)} \right) ,
\\
C := \left(\frac{n+1}{2n},\frac{n-1}{2n}\right),
\end{gather*}
and the trapezoidal region
\begin{gather*}
\begin{split}
\mathcal{T}_n 
:=&\,
\left\{
\left(\frac{1}{p},\frac{1}{q}\right) \in \mathcal{Q}
\colon
\frac{2}{n+1} \le \frac{1}{p}-\frac{1}{q} \le \frac{2}{n} , \;
\frac{1}{p} > \frac{n+1}{2n}, \;
\frac{1}{q} < \frac{n-1}{2n}
\right\}
\\
=&\,
[A,B,B',A']\setminus\{[A,B]\cup[A',B']\}
\end{split}
\end{gather*}
where $\mathcal{Q}$ is the square $[0,1]\times[0,1]$ and, for any finite set of points $\{p_1,\dots,p_k\} \subseteq \mathcal{Q}$, we denote with $[p_1,\dots,p_k]$ its convex hull.

In the $2$-dimensional case we define
\begin{gather*}
B := \left(\frac{3}{4} , \frac{1}{12}\right) ,
\quad
B' := \left(\frac{11}{12} , \frac{1}{4}\right),
\quad
B_0 := \frac{B+B'}{2} = \left(\frac{5}{6} , \frac{1}{6}\right),
\\
A_0 := (1,0),
\quad
C := \left(\frac{3}{4},\frac{1}{4}\right),
\quad
D := \left(\frac{3}{4} , 0 \right),
\quad
D' := \left(1 , \frac{1}{4}\right)
\end{gather*}
and the diamond region
\begin{gather*}
\begin{split}
\mathcal{T}_2
:=&\,
\left\{
\left(\frac{1}{p},\frac{1}{q}\right) \in \mathcal{Q}
\colon
\frac{2}{3} \le \frac{1}{p}-\frac{1}{q} <1, \;
\frac{3}{4} < \frac{1}{p} \le 1, \;
0 \le \frac{1}{q} < \frac{1}{4}
\right\}
\\
=&\,
[B,D,A_0,D',B']\setminus\{[B,D]\cup\{A_0\}\cup[B',D']\}.
\end{split}
\end{gather*}

\begin{lemma}\label{resest-T1}
	Let $z \in \C\setminus[0,\infty)$. 
	If $n=1$, then
	\begin{equation*}
	\n{(-\Delta-z)^{-1}}_{L^1 \to L^\infty} \le \frac{1}{2} |z|^{-1/2}.
	\end{equation*}
	If $n\ge2$, there exists a constant $C>0$ independent on $z$ such that:
	\begin{enumerate}[label=(\roman*)]
		\item if $(1/p,1/q) \in \mathcal{T}_n$, then
		\begin{equation}\label{KRSG}
		\n{(-\Delta-z)^{-1}}_{L^p \to L^q} \le C |z|^{-1+\frac{n}{2}\left(\frac{1}{p}-\frac{1}{q}\right)};
		\end{equation}
		\item if $(1/p,1/q) \in \{B,B'\}$ or if, when $n\ge3$, $(1/p,1/q) \in \{A,A'\}$, then the restricted weak-type estimate
		\begin{equation*}
		\n{(-\Delta-z)^{-1}}_{L^{p,1} \to L^{q,\infty}} \le C |z|^{-1+\frac{n}{2}\left(\frac{1}{p}-\frac{1}{q}\right)}
		\end{equation*}
		holds true.
	\end{enumerate}
\end{lemma}

The $1$-dimensional estimate immediately follows from the explicit representation for the kernel of the Laplacian resolvent, i.e. 
$$(-\Delta-z)^{-1} u(x) = \int_{-\infty}^{+\infty} \frac{i}{2\sqrt{z}} e^{i\sqrt{z}|x-y|} u(y) dy,$$
and from the Young's inequality. This estimate was firstly applied to obtain an eigenvalues localization for the Schr\"odinger operator by Abramov, Aslanyan and Davies \cite{AbramovAslanyanDavies01}.

The estimate in Lemma\til\ref{resest-T1}.(i) has been proved true on the open segment $(A,A')$ and on the conjugate segment $[A_0,B_0]$ in Lemma\til2.2.(b) and Theorem\til2.3 of the celebrated paper \cite{KenigRuizSogge87} by Kenig, Ruiz and Sogge. 
From here comes out the adjective \lq\lq uniform'' with which these kind of estimates are known (even if the multiplicative factor in general shows a dependence on $z$): the main result in \cite{KenigRuizSogge87} concerns the exponents on the segment $(A,A')$, on which the exponent in the factor $|z|^{-1+(1/p-1/q)n/2}$ is indeed equal to zero. Nowadays, the term \lq\lq uniform'' is generally used when the multiplicative factor is bounded for large value of $|z|$, which is relevant if we want to localize the eigenvalues in compact sets.

The estimate \eqref{KRSG} was then proved true on the optimal range $(1/p,1/q) \in \mathcal{T}_n$ by Guti\'errez in Theorem\til6 of \cite{Gutierrez04}. In this work the author proved also the inequality at Lemma\til\ref{resest-T1}.(ii) on the endpoints $B$ and $B'$, whereas the proof for the endpoints $A$ and $A'$ was recently given by Ren, Xi and Zhang in \cite{RenXiZhang18}.

It should be noted that both the works \cite{KenigRuizSogge87} and \cite{Gutierrez04} assume $n\ge3$. The $2$-dimensional case seems to have been gone quietly in the literature, nevertheless the arguments in the aforementioned papers can be quite smoothly extended in dimension $n=2$. This has been observed firstly in Frank \cite{Frank11} concerning the Kenig, Ruiz and Sogge's result, and by Kwon and Lee \cite{KwonLee20} about the work by Guti\'errez. 

Now, one question arises naturally: does estimates similar to \eqref{KRSG} hold outside the region $\mathcal{T}_n$? Well yes, but actually no. 
The range of exponents stated in the above theorem is optimal: estimates \eqref{KRSG} does not hold true if $(1/p,1/q)$ lies outside $\mathcal{T}_n$. For $n\ge3$, the constrains $\frac{1}{p} > \frac{n+1}{2n}$ and $\frac{1}{q}< \frac{n-1}{2n}$ are due to considerations from the theory of the Bochner-Riesz operators of negative orders, the condition $\frac{1}{p}-\frac{1}{q} \ge \frac{2}{n+1}$ comes from the Knapp counterexample and finally $\frac{1}{p}-\frac{1}{q} \le \frac{2}{n}$ follows by an argument involving the Littlewood-Paley projection. For details on this discussion we refer to\til\cite{KwonLee20} (and to\til\cite{KenigRuizSogge87}).

Nonetheless, we can still extend the region of the estimates if we sacrifice something. This is the main theme of the paper \cite{KwonLee20} by Kwon and Lee, where they conjecture that, for $n\ge2$ and $z\in\C\setminus[0,\infty)$, the relation
\begin{equation}\label{KLapprox}
\n{(-\Delta-z)^{-1}}_{L^p \to L^q} \approx |z|^{-1+\frac{d}{2}\left(\frac{1}{p}-\frac{1}{q}\right)} \left(\frac{|z|}{\dist(z,[0,\infty))}\right)^{\gamma(n,p,q)} 
\end{equation}
with
\begin{equation}\label{gamma(n,p,q)}
\gamma(n,p,q) :=
\max
\left\{
0,\,
1-\frac{n+1}{2}\left(\frac{1}{p}-\frac{1}{q}\right),\,
\frac{n+1}{2}-\frac{n}{p},\,
\frac{n}{q}-\frac{n-1}{2}
\right\}
\end{equation}
should hold on the \lq\lq stripe''
\begin{gather}\label{regionS}
\mathcal{S}
:=
\left\{
\left( \frac{1}{p} , \frac{1}{q} \right) \in \mathcal{Q} 
\colon
0 \le \frac{1}{p}-\frac{1}{q} \le \frac{2}{n}
\right\}
\setminus
\mathcal{S}_0
\end{gather}
where 
\begin{gather*}
\mathcal{S}_0 :=
\left\{
\begin{aligned}
&[A,B] \cup [A',B']
\cup
\left[E,E_0 \right)
\cup
\left(E_0,E' \right]
\cup
\left\{F\right\}
\cup
\left\{F'\right\}
&&\text{if $n\ge3$,}
\\
&[B,D]
\cup[B',D']
\cup
\left[E,E_0 \right)
\cup
\left(E_0,E' \right]
\cup
\left\{A_0\right\}
&&\text{if $n=2$,}
\end{aligned}
\right.
\\[3pt]
E := \left( \frac{n-1}{2n},\frac{n-1}{2n} \right),
\quad
E' := \left( \frac{n+1}{2n},\frac{n+1}{2n} \right),
\quad
E_0 := \left( \frac{1}{2},\frac{1}{2} \right),
\\
F := \left(\frac{2}{n},0\right),
\quad
F' := \left(1,\frac{n-2}{n}\right).
\end{gather*}
The symbol $A \approx B$ in \eqref{KLapprox} means that there exists an absolute constant, independent on $z$, such that $ C^{-1} B \le A \le C B$.

Observe that the region $\mathcal{S}$ contains in particular $\mathcal{T}_n$, on which $\gamma(n,p,q)=0$ as one can naturally expect in light of the Kenig-Ruiz-Sogge-Guti\'errez inequalities.
In their work, Kwon and Lee prove their conjecture to be indeed true, making exception of the upper bound 
implicitly contained in \eqref{KLapprox}
on the region 
\begin{equation}\label{regionR}
	\widetilde{\mathcal{R}} := 
	\begin{cases}
	\varnothing &\text{if $n=2$,}
	\\
	\mathcal{R} \cup \mathcal{R}' &\text{if $n\ge3$,}
	\end{cases}
\end{equation}
where 
\begin{equation*}
	\mathcal{R} := [P_*,P_\circ,E_0]\setminus\{E_0\},
	\quad
	\mathcal{R}' := [P'_*,P'_\circ,E_0]\setminus\{E_0\},
\end{equation*}
and the endpoints are defined by
\begin{gather*}
P_* := \left(\frac{1}{p_*},\frac{1}{p_*}\right),
\qquad
P'_* := \left(1-\frac{1}{p_*},1-\frac{1}{p_*}\right),
\qquad
\frac{1}{p_*}
:=
\left\{
\begin{aligned}
& \frac{3(n-1)}{2(3n+1)} &&\text{if $n$ is odd,}
\\
& \frac{3n-2}{2(3n+2)} &&\text{if $n$ is even,}
\end{aligned}
\right.
\\
P_\circ := \left(\frac{1}{p_\circ},\frac{1}{q_\circ}\right),
\qquad
P'_\circ := \left(1-\frac{1}{q_\circ},1-\frac{1}{p_\circ}\right),
\\
\frac{1}{p_\circ}
:=
\left\{
\begin{aligned}
& \frac{(n+5)(n-1)}{2(n^2+4n-1)} &&\text{if $n$ is odd,}
\\
& \frac{n^2+3n-6}{2(n^2+3n-2)} &&\text{if $n$ is even,}
\end{aligned}
\right.
\qquad
\frac{1}{q_\circ}
:=
\left\{
\begin{aligned}
& \frac{(n+3)(n-1)}{2(n^2+4n-1)} &&\text{if $n$ is odd,}
\\
& \frac{(n-1)(n+2)}{2(n^2+3n-2)} &&\text{if $n$ is even.}
\end{aligned}
\right.
\end{gather*}

Let us gather the above results by Kwon and Lee \cite{KwonLee20} in the following

\begin{lemma}\label{resest-T3}
	Let $n\ge2$ and $z \in \C\setminus[0,\infty)$. There exists a constant $K>0$ independent on $z$ such that:
	\begin{enumerate}[label=(\roman*)]
		\item if $(1/p,1/q) \in \mathcal{S}$, then
		\begin{equation*}
		\n{(-\Delta-z)^{-1}}_{L^p \to L^q} \ge K^{-1} |z|^{-1+\frac{n}{2}\left(\frac{1}{p}-\frac{1}{q}\right)} \left(\frac{|z|}{\dist(z,[0,\infty))}\right)^{\gamma(n,p,q)};
		\end{equation*}
		\item if $(1/p,1/q) \in \mathcal{S}\setminus\widetilde{\mathcal{R}}$, then
		\begin{equation*}
		\n{(-\Delta-z)^{-1}}_{L^p \to L^q} \le K |z|^{-1+\frac{n}{2}\left(\frac{1}{p}-\frac{1}{q}\right)} \left(\frac{|z|}{\dist(z,[0,\infty))}\right)^{\gamma(n,p,q)}.
		\end{equation*}
	\end{enumerate}
	The regions $\mathcal{S}$ and $\widetilde{\mathcal{R}}$ are described in \eqref{regionS} and \eqref{regionR} respectively, while $\gamma(n,p,q)$ is defined in \eqref{gamma(n,p,q)}.
\end{lemma}

The analysis of Kwon and Lee pictures quite clearly the situation outside the so-called \lq\lq uniform boundedness range'' $\mathcal{T}_n$: we can still have $L^p-L^q$ inequalities so long as the factor depending on $z$ explodes when $\Im z\to0^\pm$, and this can not be improved. 
If we want to apply these estimates in the eigenvalues localization problem, this means that we can not obtain the eigenvalues confined in a compactly supported region of the complex plane, but in a set containing the continuous spectrum of the unperturbed operator.

In this optic, one can instead try to save the uniformity of the estimates, in the sense that the factor depending on $z$ should be uniformly bounded for $|z|$ sufficiently large. In this way, we can again hope to get the eigenvalues confined inside compact regions. This can be indeed obtained on a smaller region respect to $\mathcal{S}$ if we restrict ourself on considering radial functions.

Define, for $n\ge2$, the open triangle
\begin{equation*}
\begin{split}
\mathcal{P}
:=&\,
\left\{
\left(\frac{1}{p},\frac{1}{q}\right) \in \mathcal{Q}
\colon
\frac{1}{n} <
\frac{1}{p}-\frac{1}{q} < \frac{2}{n+1} , \;
\frac{1}{p} > \frac{n+1}{2n}, \;
\frac{1}{q} < \frac{n-1}{2n}
\right\}
\\
=&\,
[B,C,B']\setminus\{[B,B']\cup[B,C]\cup[C,B']\}.
\end{split}
\end{equation*}
Recall the radial-angular spaces \eqref{eq:radial-angular-spaces} defined in the Introduction and their norms. Adopting the terminology and notations of \cite{BarceloRuizVega97} and \cite{FrankSimon17}, we introduce also the radial Mizohata-Takeuchi norm 
\begin{equation*}
\n{w}_{\MT}
:=
\sup_{R>0} \int_{R}^{\infty}
\frac{r}{\sqrt{r^2-R^2}} 
\n{w(r\,\cdot)}_{L^\infty(\S^{n-1})}
dr
\end{equation*}
and we say that $w \in \MT$ if $\n{w}_{\MT} <\infty$. 

\begin{lemma}\label{resest-T2}
	Let $n\ge2$ and $z \in \C\setminus[0,\infty)$. There exists a constant $K>0$ independent on $z$ such that:
	\begin{enumerate}[label=(\roman*)]
		\item if $(1/p,1/q) \in (C,B_0)$, then
		\begin{equation*}
		\n{(-\Delta-z)^{-1}}_{L_\rho^{p}L_\theta^2 \to L_\rho^{q}L_\theta^2} \le K |z|^{-1-\frac{n}{2}+\frac{n}{p}};
		\end{equation*}
		\item if $(1/p,1/q)=C$, then
		\begin{align}
		\label{eq:BRV-C}
		\n{(-\Delta-z)^{-1}}_{L_\rho^{2n/(n+1),1}L_\theta^2 \to L_\rho^{2n/(n-1),\infty}L_\theta^2} &\le K |z|^{-1/2},
		\\
		\label{eq:BRV-C-nabla}
		\n{\nabla(-\Delta-z)^{-1}}_{L_\rho^{2n/(n+1),1}L_\theta^2 \to L_\rho^{2n/(n-1),\infty}L_\theta^2} &\le K 
		.
		\end{align}
	\end{enumerate}
	If in particular $u \in L^p(\R^n)$ is a radial function, then
	\begin{equation*}
	\n{(-\Delta-z)^{-1} u}_{L^q} \le K |z|^{-1+\frac{n}{2}\left(\frac{1}{p}-\frac{1}{q}\right)}
	\n{u}_{L^p}
	\end{equation*}
	for any $(1/p,1/q)\in\mathcal{P}$, and
	\begin{align*}
	\n{(-\Delta-z)^{-1} u}_{L^{2n/(n-1),\infty}} &\le K |z|^{-1/2} \n{u}_{L^{2n/(n+1),1}}
	\\
	\n{\nabla(-\Delta-z)^{-1} u}_{L^{2n/(n-1),\infty}} &\le K \n{u}_{L^{2n/(n+1),1}}
	\end{align*}
	in the case $(1/p,1/q)=C$.
\end{lemma}

The result in Lemma\til\ref{resest-T2}.(i) is stated in Theorem\til4.3 by Frank and Simon \cite{FrankSimon17}. Instead, the case of the endpoint $C$ is essentially due to Theorem\til1.(b) and Theorem\til2 by Barcelo, Ruiz and Vega \cite{BarceloRuizVega97}. 
Indeed, let us consider firstly the estimate for $(-\Delta-z)^{-1}$. Observe that, by H\"older's inequality and by duality, the estimate \eqref{eq:BRV-C} is equivalent to 
\begin{equation}\label{eq:BRV-w}
\n{w_1^{1/2} (-\Delta-z)^{-1} w_2^{1/2} u}_{L^2} \le K |z|^{-1/2} 
\n{w_1}_{L^{n,1}_\rho L^\infty_\theta}^{1/2}
\n{w_2}_{L^{n,1}_\rho L^\infty_\theta}^{1/2}
\n{u }_{L^2}
\end{equation}
for any $w_1,w_2 \in L^{n,1}_\rho L^\infty_\theta$. In fact, that \eqref{eq:BRV-C} implies \eqref{eq:BRV-w} is obvious by H\"older's inequality for Lorentz spaces. Conversely, we have that
\begin{align*}
\n{(-\Delta-z)^{-1} w_2^{1/2} u}_{L_\rho^{2n/(n-1),\infty} L_\theta^2} 
&= \sup_{0 \neq w_1 \in L_\rho^{n,1} L_\theta^\infty}
\frac{\n{w_1^{1/2} (-\Delta-z)^{-1} w_2^{1/2} u}_{L^2}}{\n{w_1^{1/2}}_{L_\rho^{2n,2} L_\theta^\infty}}
\\
&\le
K
|z|^{-1/2} 
\n{w_2}_{L_\rho^{n,1} L_\theta^\infty}
\n{u}_{L^2},
\end{align*}
that is to say that, for any fixed $w\in L_\rho^{n,1} L_\theta^\infty$, the operator $(-\Delta-z)^{-1} w^{1/2}$ is bounded from $L^2$ to $L_\rho^{2n/(n-1),\infty} L_\theta^2$ with norm 
$$\n{(-\Delta-z)^{-1} w^{1/2}}_{L^2 \to L_\rho^{2n/(n-1),\infty} L_\theta^2}
\le
K |z|^{-1/2} \n{w}_{L_\rho^{n,1} L_\theta^\infty}.$$
By duality this implies that the operator $w^{1/2}(-\Delta-z)^{-1}$ is bounded from $L_\rho^{2n/(n+1),1} L_\theta^2$ to $L^2$ with norm 
$$\n{ w^{1/2} (-\Delta-z)^{-1}}_{L_\rho^{2n/(n+1),1} L_\theta^2 \to L^2}
\le
K |z|^{-1/2} \n{w}_{L_\rho^{n,1} L_\theta^\infty},$$
from which we finally get
\begin{align*}
\n{(-\Delta-z)^{-1} u}_{L_\rho^{2n/(n-1),\infty} L_\theta^2} 
&= \sup_{0 \neq w_1 \in L_\rho^{n,1} L_\theta^\infty}
\frac{\n{w_1^{1/2} (-\Delta-z)^{-1} u}_{L^2}}{\n{w_1^{1/2}}_{L_\rho^{2n,2} L_\theta^\infty}}
\\
&\le
K |z|^{-1/2} 
\n{u}_{L_\rho^{2n/(n+1),1} L_\theta^2}.
\end{align*}
From Barcelo, Ruiz and Vega \cite{BarceloRuizVega97} we have that
\begin{equation}\label{eq:BRV-MT}
	\n{w_1^{1/2} (-\Delta-z)^{-1} w_2^{1/2} u}_{L^2} 
	\le K
	|z|^{-1/2} 
	\n{w_1}_{\MT}^{1/2}
	\n{w_2}_{\MT}^{1/2}
	\n{u }_{L^2}
\end{equation} 
which implies \eqref{eq:BRV-w}. Indeed, we can replace the $\MT$ norm with the $L_\rho^{n,1} L_\theta^\infty$ norm since, as proved in equation (4.2) of \cite{FrankSimon17}, the embedding 
$$L_\rho^{n,1} L_\theta^\infty \hookrightarrow \MT$$ 
holds true (cf. Theorem\til4.4 in \cite{FrankSimon17}).
To be precise, 
equation \eqref{eq:BRV-MT} is proved in \cite{BarceloRuizVega97} for $w_1=w_2 \in\MT$, but the possibility of choosing two different weights follows easily from their proof (see Proposition\til2 of the same paper).

	Consider now the estimate for $\nabla (-\Delta-z)^{-1}$ on the endpoint $C$.
	From Theorem\til2 in \cite{BarceloRuizVega97} we have that
	\begin{equation}\label{eq:nabla-BRV}
		\n{ v}_{L^2} \le K
		\n{w}_{\MT}
		\n{ w^{1/2} \nabla (-\Delta-z) w^{-1/2} v }_{L^2}
	\end{equation}
	for $z\ge0$. Supposing this inequality true for any complex number $z$, we can then obtain estimate \eqref{eq:BRV-C-nabla} following the same argument as above. The fact that \eqref{eq:nabla-BRV} is true everywhere on the complex plane is implicit in the proof given by Barcelo, Ruiz and Vega. Indeed, the proof of Theorem\til2 at pages 373--374 of \cite{BarceloRuizVega97} is still valid for any real $z$. Then, the argument based on the Phragm\'en-Lindel\"of principle exploited at page 373 to prove Theorem\til1.(b) can be adapted also to this situation, proving \eqref{eq:nabla-BRV} for any $z\in\C$.

Finally, for radial functions the radial-angular norms \eqref{eq:radial-angular} from \cite{FrankSimon17} reduce simply to the Lebesgue and Lorentz norms. Real interpolation between the estimates on the open segment $(C,B_0)$ and the ones on the open segment $(B,B')$ coming from Lemma\til\ref{resest-T1} prove the assertion on $\mathcal{P}$ for radial functions.

	Thus ends our recap on the Schr\"odinger resolvent estimates. The results in Lemmata\til\ref{resest-T1},\til\ref{resest-T3} and \ref{resest-T2} are visually summarized in Figure\til\ref{fig:resest}. 
	We conclude this section with an direct corollary of Lemma\til\ref{resest-T2} concerning the free Dirac resolvent.

	\begin{corollary}\label{cor:radial-dirac}
		Let $n\ge2$, $m\ge0$ and 
		$z\in\C \setminus \{ (-\infty,-m] \cup [m,+\infty) \}$. 
		There exists a constant $K>0$ independent on $z$ such that
		\begin{equation*}
			\n{(\D_{m}-z)^{-1}}_{L_\rho^{2n/(n+1),1}L_\theta^2 \to L_\rho^{2n/(n-1),\infty}L_\theta^2} 
			\le K 
			\left[ 1 + \left\lvert \frac{z+m}{z-m} \right\rvert^{\frac{\sgn\Re z}{2}} \right]
		\end{equation*}
		and in particular, if $u\in L^{\frac{2n}{n+1},1}(\R^n)$ is a radial function, then
		\begin{equation*}
			\n{(\D_{m}-z)^{-1} u}_{L^{2n/(n-1),\infty}} 
			\le K 
			\left[ 1 + \left\lvert \frac{z+m}{z-m} \right\rvert^{\frac{\sgn\Re z}{2}} \right]
			\n{u}_{L^{2n/(n+1),1}}.
		\end{equation*}
	\end{corollary}
	
	\begin{proof}
		By the identity \eqref{eq:dirac-resolvent-identity} and the estimates \eqref{eq:BRV-C}--\eqref{eq:BRV-C-nabla}, 
		it is immediate to get
		\begin{equation*}
			\begin{split}
				\n{(\D_{m}-z)^{-1} u}_{L^{2n/(n-1),\infty}} 
				\le&\,
				\n{\sum_{k=1}^n \alpha_k \partial_k (-\Delta+m^2-z^2)^{-1} u }_{L^{2n/(n-1),\infty}} 
				\\
				&+ 
				\n{(m\alpha_{n+1} + zI_N) (-\Delta+m^2-z^2)^{-1} u}_{L^{2n/(n-1),\infty}} 
				\\
				\le&\,
				\sqrt{n} \n{\nabla (-\Delta+m^2-z^2)^{-1} u}_{L^{2n/(n-1),\infty}}
				\\
				&+ 
				\max\{|z+m|,|z-m|\}
				\n{(-\Delta+m^2-z^2)^{-1} u}_{L^{2n/(n-1),\infty}} 
				\\
				\le&\,
				K 
				\left[ 1 + \left\lvert \frac{z+m}{z-m} \right\rvert^{\frac{\sgn\Re z}{2}} \right]
				\n{u}_{L^{2n/(n+1),1}}
			\end{split}
		\end{equation*}
	and hence the claimed inequalities.
	\end{proof}


\begin{figure}[!htb]
	\begin{subfigure}[b]{0.5\textwidth}
		\centering
		\scriptsize		
		\resizebox{\linewidth}{!}{
			\begin{tikzpicture}[scale=7]			
				\foreach \n in {2}
				{
					\coordinate (D) at ( {3/4} , {0} ) {};
					\coordinate (D') at ( {1} , {1/4} ) {};
					\coordinate (B) at ( {(\n+1)/(2*\n)} , {(\n-1)*(\n-1)/(2*\n*(\n+1))} ) {};
					\coordinate (B') at ( {(\n*\n+4*\n-1)/(2*\n*(\n+1))} , {(\n-1)/(2*\n)} ) {};
					\coordinate (A0) at ( {(\n+2)/(2*\n)} , {(\n-2)/(2*\n)} ) {};
					\coordinate (B0) at ( {(\n+3)/(2*(\n+1))} , {(\n-1)/(2*(\n+1))} ) {};
					\coordinate (C) at ( {(\n+1)/(2*\n)} , {(\n-1)/(2*\n)} ) {};
					\coordinate (E) at ( {(\n-1)/(2*\n)} , {(\n-1)/(2*\n)} ) {};
					\coordinate (E') at ( {(\n+1)/(2*\n)} , {(\n+1)/(2*\n)} ) {};
					\coordinate (E0) at ( {1/2} , {1/2} ) {};

					\path [fill=Goldenrod!60] 
					(0,0) -- (1,1) -- (1,0) -- cycle;
					\path [fill=Cyan!60]
					(A0) -- (D) -- (B) -- (B') -- (D') -- cycle;
					\path [fill=Red!60]
					(B) -- (C) -- (B') -- cycle;

					\draw [dotted] (0,1) -- (1,0);
					
					\draw [dashed] (D)node[below left]{$D$} 
					-- (B)node[left]{$B$}
					-- (C)node[above left]{$C$}
					-- (B')node[above]{$B'$}
					-- (D')node[above right]{$D'$}
					;
					
					\draw (B) -- (B0)node[right]{$B_0$} -- (B');
					\draw (0,0) -- (E);
					\draw [dashed] (E)node[left]{$E$} -- (E0)node[above left]{$E_0$} -- (E')node[above]{$E'$};
					\draw (E') -- (1,1);

					{\footnotesize
						\draw (9.25/10,.75/10) node{$\mathcal{T}_2$};
						\draw (6/10,4/10) node{$\widetilde{S}$};
						\draw (8/10,2/10) node{$\mathcal{P}$};
					}

					\draw [latex-latex] 
					(0,1.1)node[left]{$\frac{1}{q}$} -- (0,0) node[below left]{$(0,0)$} -- (1.1,0) node[below]{$\frac{1}{p}$};
					\draw (0,1) node[left]{$(0,1)$} -- (1,1) -- (1,0) node[below]{$A_0$};

					\node at (D) [shape=circle, inner sep=1pt, draw, fill=white]{};
					\node at (D') [shape=circle, inner sep=1pt, draw, fill=white]{};
					\node at (B) [shape=circle, inner sep=1pt, draw, fill=white]{};
					\node at (B') [shape=circle, inner sep=1pt, draw, fill=white]{};
					\node at (C) [shape=circle, inner sep=1pt, draw, fill=white]{};
					\node at (B0) [shape=circle, inner sep=1pt, draw, fill]{};
					\node at (E0) [shape=circle, inner sep=1pt, draw, fill]{};
					\node at (E) [shape=circle, inner sep=1pt, draw, fill=white]{};
					\node at (E') [shape=circle, inner sep=1pt, draw, fill=white]{};
					\node at (A0) [shape=circle, inner sep=1pt, draw, fill=white]{};
				}		
			\end{tikzpicture}
		}
		\caption{Case $n=2$.}
	\end{subfigure}
	\begin{subfigure}[b]{0.5\textwidth}
		\centering
		\scriptsize
		\resizebox{\linewidth}{!}{
			\begin{tikzpicture}[scale=7]
				\foreach \n in {4}
				{
					
					\coordinate (A) at ( {(\n+1)/(2*\n)} , {(\n-3)/(2*\n)} ) {};
					\coordinate (A') at ( {(\n+3)/(2*\n)} , {(\n-1)/(2*\n)} ) {};
					\coordinate (B) at ( {(\n+1)/(2*\n)} , {(\n-1)*(\n-1)/(2*\n*(\n+1))} ) {};
					\coordinate (B') at ( {(\n*\n+4*\n-1)/(2*\n*(\n+1))} , {(\n-1)/(2*\n)} ) {};
					\coordinate (A0) at ( {(\n+2)/(2*\n)} , {(\n-2)/(2*\n)} ) {};
					\coordinate (B0) at ( {(\n+3)/(2*(\n+1))} , {(\n-1)/(2*(\n+1))} ) {};
					\coordinate (C) at ( {(\n+1)/(2*\n)} , {(\n-1)/(2*\n)} ) {};
					\coordinate (E) at ( {(\n-1)/(2*\n)} , {(\n-1)/(2*\n)} ) {};
					\coordinate (E') at ( {(\n+1)/(2*\n)} , {(\n+1)/(2*\n)} ) {};
					\coordinate (E0) at ( {1/2} , {1/2} ) {};
					\coordinate (F) at ( {2/\n} , {0} ) {};
					\coordinate (F') at ( {1} , {(\n-2)/\n} ) {};
					%
					
					
					\coordinate (P*) at ( {(3*\n-2)/2/(3*\n+2)} , {(3*\n-2)/2/(3*\n+2)} ) {};
					\coordinate (P*') at ( {1-(3*\n-2)/2/(3*\n+2)} , {1-(3*\n-2)/2/(3*\n+2)} ) {};
					\coordinate (Po) at ( {(\n*\n+3*\n-6)/2/(\n*\n+3*\n-2)} , {(\n-1)*(\n+2)/2/(\n*\n+3*\n-2)} ) {};
					\coordinate (Po') at ( {1-(\n-1)*(\n+2)/2/(\n*\n+3*\n-2)} , {1-(\n*\n+3*\n-6)/2/(\n*\n+3*\n-2)} ) {};

					\path [pattern color=Goldenrod!80, pattern=crosshatch dots] 
					(F) -- (0,0) -- (1,1) -- (F') -- cycle;
					\path [fill=Goldenrod!60] 
					(F) -- (0,0) -- (P*) -- (Po) -- (E0) -- (Po') -- (P*') -- (1,1) -- (F') -- cycle;
					\path [fill=Cyan!60]
					(A) -- (B) -- (B') -- (A') -- cycle;
					\path [fill=Red!60]
					(B) -- (C) -- (B') -- cycle;

					\draw [dotted] (0,1) -- (1,0);
					
					\draw [dashed] (A)node[left]{$A$} 
					-- (B)node[left]{$B$}
					-- (C)node[above left]{$C$}
					-- (B')node[above]{$B'$}
					-- (A')node[above]{$A'$}
					;
					
					\draw (B) -- (B0)node[right]{$B_0$} -- (B');
					\draw (A) -- (A0)node[right]{$A_0$} -- (A');
					\draw (F)node[below]{$F$} -- (A);
					\draw (A') -- (F')node[right]{$F'$};
					\draw (0,0) -- (P*);
					\draw [dashed] (P*)node[left]{$P_*$} -- (Po)node[below]{$P_\circ$} -- (E0)node[above left]{$E_0$} -- (Po')node[right]{$P_\circ'$} -- (P*')node[above]{$P_*'$};
					\draw (P*') -- (1,1);

					{\footnotesize
						\draw ( {(\n+1)/(2*\n)/2+(\n+2)/(2*\n)/2} , {(\n-1)*(\n-1)/(2*\n*(\n+1))} ) node{$\mathcal{T}_n$};
						\draw ( {(3*\n-2)/2/(3*\n+2)} , {(\n-3)/(2*\n)} ) node{$\widetilde{\mathcal{S}}$};
						\draw ( {(\n+1)/(2*\n)/2 + (\n+3)/(2*(\n+1))/2} , {(\n-1)/(2*\n)/2 + (\n-1)/(2*(\n+1))/2} ) node{$\mathcal{P}$};		
						\draw ( {(3*\n-2)/2/(3*\n+2)/2 + 1/4} , {(3*\n-2)/2/(3*\n+2)/2 +1/4} ) node{$\mathcal{R}$};	
						\draw ( {1-(3*\n-2)/2/(3*\n+2)/2 - 1/4} , {1-(3*\n-2)/2/(3*\n+2)/2 -1/4} ) node{$\mathcal{R'}$};	
					}

					\draw [latex-latex] 
					(0,1.1)node[left]{$\frac{1}{q}$} -- (0,0) node[below left]{$(0,0)$} -- (1.1,0) node[below]{$\frac{1}{p}$};
					\draw (0,1) node[left]{$(0,1)$} -- (1,1) -- (1,0) node[below]{$(1,0)$};

					\node at (A) [shape=circle, inner sep=1pt, draw, fill=white]{};
					\node at (A') [shape=circle, inner sep=1pt, draw, fill=white]{};
					\node at (B) [shape=circle, inner sep=1pt, draw, fill=white]{};
					\node at (B') [shape=circle, inner sep=1pt, draw, fill=white]{};
					\node at (C) [shape=circle, inner sep=1pt, draw, fill=white]{};
					\node at (A0) [shape=circle, inner sep=1pt, draw, fill]{};
					\node at (B0) [shape=circle, inner sep=1pt, draw, fill]{};
					\node at (E0) [shape=circle, inner sep=1pt, draw, fill]{};
					\node at (F) [shape=circle, inner sep=1pt, draw, fill=white]{};
					\node at (F') [shape=circle, inner sep=1pt, draw, fill=white]{};
					\node at (P*) [shape=circle, inner sep=1pt, draw, fill=white]{};
					\node at (Po) [shape=circle, inner sep=1pt, draw, fill=white]{};
					\node at (P*') [shape=circle, inner sep=1pt, draw, fill=white]{};
					\node at (Po') [shape=circle, inner sep=1pt, draw, fill=white]{};
					
				}
				
			\end{tikzpicture}
		}
		\caption{Case $n\ge3$.}
	\end{subfigure}

	\caption{In this picture we visualize the many regions and endpoints appearing in Section\til\ref{sec:estimates}. The Kenig-Ruiz-Sogge-Guti\'errez region $\mathcal{T}_n$ from Lemma\til\ref{resest-T1} is highlighted in blue, while in red we show the triangle $\mathcal{P}$ from Lemma\til\ref{resest-T2} about the estimates for radial functions. Finally, the yellow region $\widetilde{\mathcal{S}}$ is such that $\mathcal{S} \setminus\widetilde{\mathcal{R}} =  \widetilde{\mathcal{S}} \cup \mathcal{P} \cup \mathcal{T}_n $, where $\mathcal{S}$ is the Kwon-Lee region interested by Lemma\til\ref{resest-T3} and $\widetilde{\mathcal{R}}=\mathcal{R}\cup\mathcal{R}'$ is pictured dotted.}
	\label{fig:resest}
\end{figure}
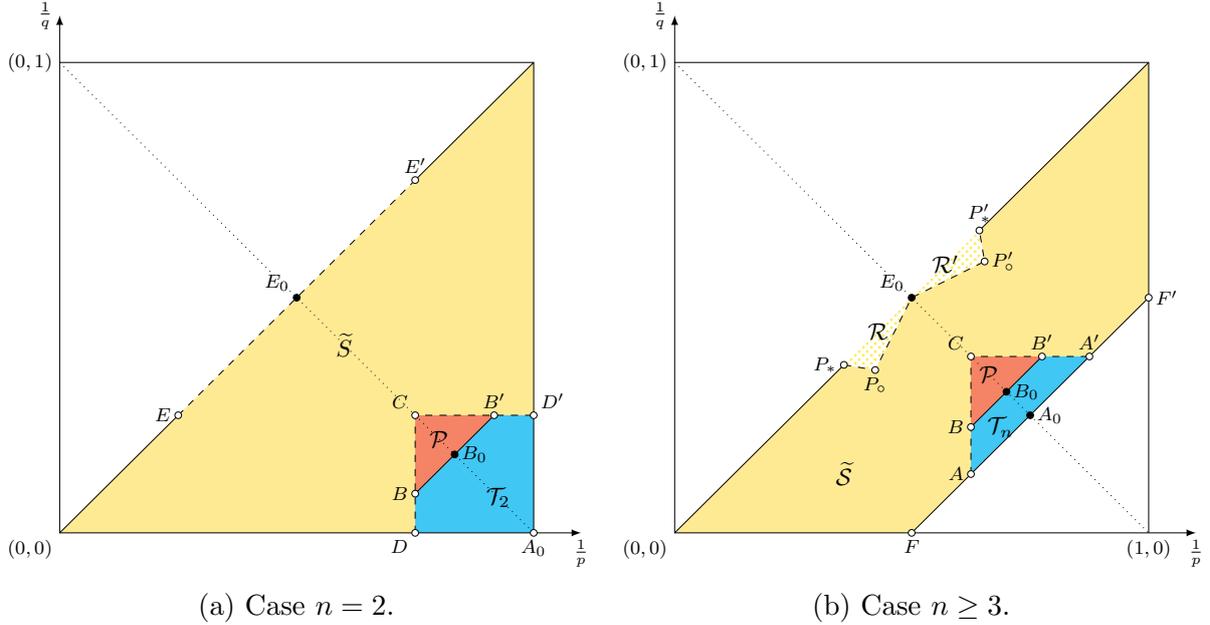


%
Let us combine now the estimates above with the Birman-Schwinger principle to get our claimed results.


\section{Proofs for the Theorems}\label{sec:BS}

This section is splitted in two parts: in the first one we recall the technicalities of the Birman-Schwinger principle and we properly define an operator perturbed by a factorizable potential; in the second part we complete the (very straightforward) computations to prove our theorems in Section\til\ref{sec:main}.

\subsection{The Birman-Schwinger principle}

The number of works in the literature which make use of the Birman-Schwinger principle is huge. Anyway, recently an abstract analysis has been carried out by Hansmann and Krej\v ci\v r\'ik in \cite{HansmannKrejcirik2020}.
Here we sketch in a synthetic way their approach. We refer to their work for more results, the background and a complete discussion.

First of all, let us state the necessary hypothesis.

\begin{hypothesis}\label{ass}
	Let $\H$ and $\H'$ be complex separable Hilbert spaces, $H_0$ be a selfadjoint operator in $\H$ and $|H_0| := (H_0^2)^{1/2}$ its absolute value. Also, let $\A \colon \dom(\A) \subseteq \H \to \H'$ and $\B \colon \dom(\B) \subseteq \H \to \H'$ be linear operators such that $\dom(|H_0|^{1/2}) \subseteq \dom(\A) \cap \dom(\B)$.
	
	We assume that for some (and hence for all) $b>0$ the operators $\A(|H_0|+b)^{-1/2}$ and $\B(|H_0|+b)^{-1/2}$
	are bounded and linear from $\H$ to $\H$.
\end{hypothesis}

Defined $G_0 := |H_0|+1$, let us consider, for any $z \in \rho(H_0)$, the \emph{Birman-Schwinger operator}
\begin{equation}\label{def:Kz}
K_z := [\A G_0^{-1/2}] [G_0 (H_0-z)^{-1}] [\B G_0^{-1/2}]^*,
\end{equation}
which is linear and bounded from $\H'$ to $\H'$.

\begin{remark}\label{rem:BS}
	The Birman-Schwinger operator defined above is a bounded extension of the operator $\A(H_0-z)\B^*$ defined on $\dom(\B^*)$, introduced in Section\til\ref{sec:main}. Hence, if $\dom(\B^*)$ is dense in $\H'$, then $K_z$ is exactly the closure of $\A(H_0-z)^{-1}\B^*$. 
\end{remark}

\begin{hypothesis}\label{ass2'}
	There exists $z_0 \in \rho(H_0)$ such that $\n{K_{z_0}}_{\H'\to\H'} <1$.
\end{hypothesis}

This last hypothesis is in fact much stronger than necessary, and can be substituted by the weaker assertion: \textit{There exists $z_0 \in \rho(H_0)$ such that $-1\not\in\sigma(K_{z_0})$}. But this relaxed hypothesis is way more difficult to check in practice, and Assumption\til\ref{ass2'} is good enough for us. 

We can now properly define the perturbed operator $H_0 + \V$ with $\V=\B^*\A$. 

\begin{lemma}
	\label{lem:H_V}
	Suppose Assumptions\til\ref{ass} and \ref{ass2'}. There exists a unique closed extension $H_\V$ of $H_0+\V$ such that $\dom(H_\V) \subseteq \dom(|H_0|^{1/2})$ and we have the following representation formula:
	\begin{equation*}
	(\phi, H_\V \psi)_{\H\to\H} = (G_0^{1/2} \phi, (H_0 G_0^{-1} + [\B G_0^{-1/2}]^* \A G_0^{-1/2}) G_0^{1/2}\psi)_{\H\to\H}
	\end{equation*}
	for $\phi\in\dom(|H_0|^{1/2})$, $\psi\in\dom(H_\V)$.
\end{lemma}

The above result correspond to Theorem\til5 in \cite{HansmannKrejcirik2020}, where the operator $H_\V$ is obtained via {the} pseudo-Friedrichs extension. Instead, from Theorem\til\mbox{6, 7, 8} and Corollary\til4 of \cite{HansmannKrejcirik2020} we get the abstract Birman-Schwinger principle stated below.

\begin{lemma}\label{lem:BS}
	Suppose Assumption\til\ref{ass} and \ref{ass2'}. Therefore:
	\begin{enumerate}[label=(\roman*)]
		\item if $z\in\rho(H_0)$, then $z \in \sigma_p(H_\V)$ if and only if $-1 \in \sigma_p(K_z)$;
		
		\item if $z \in \sigma_c(H_0) \cap \sigma_p(H_\V)$ and $H_\V \phi = z \phi$ for $0\neq\phi\in\dom(H_\V)$, then $\varphi:=\A\phi\neq0$ and
		\begin{equation*}
		\lim_{\e\to0^\pm} (K_{z+i\e}\varphi,\psi)_{\H'\to\H'} = -(\varphi,\psi)_{\H'\to\H'} 
		\end{equation*}
		for all $\psi\in\H'$.
	\end{enumerate}
	In particular
	\begin{enumerate}[label=(\roman*)]
		\item if $z\in\sigma_p(H_\V) \cap \rho(H_0)$, then $\n{K_z}_{\H'\to\H'} \ge 1$;
		
		\item if $z\in\sigma_p(H_\V) \cap \sigma_c(H_0)$, then $\liminf_{\e\to0^\pm} \n{K_{z+i\e}}_{\H\to\H'} \ge 1.$
	\end{enumerate}
\end{lemma} 

By strengthening Assumption\til\ref{ass2'} asking the norm of the Birman-Schwinger operator to be strictly less than $1$ uniformly respect to $z\in\rho(H_0)$, Theorem\til3 in \cite{HansmannKrejcirik2020} state the following result on the spectrum invariance for the perturbed operator.

\begin{lemma}
	\label{lem:BS2}
	Suppose Assumption\til\ref{ass} and that
	$
	\sup_{z\in\rho(H_0)} \n{K_z}_{\H'\to\H'} <1 .
	$
	Then:
	\begin{enumerate}[label=(\roman*)]
		\item $\sigma(H_0) = \sigma(H_\V)$;
		\item $\sigma_p(H_\V) \cup \sigma_r(H_\V) \subseteq \sigma_p(H_0)$ and $\sigma_c(H_0) \subseteq \sigma_c(H_V)$.
	\end{enumerate}	
	In particular, if $\sigma(H_0)=\sigma_c(H_0)$, then $\sigma(H_\V)=\sigma_c(H_\V)=\sigma_c(H_0)$.
\end{lemma}


\subsection{Computations for the proofs}\label{subsec:proof}
In our case, $\H=\H'=L^2(\R^n;\C^{N\times N})$ and $\V=\B^*\A$ is the multiplication operator in $\H$, with initial domain $\dom(\V)=C_0^\infty(\R^n;\C^{N})$, generated by a matrix-valued function $\V\colon\R^n\to\C^{N\times N}$ (with the customary abuse of notation, we use the same symbol to denote the matrix and the operator). Same thing holds for the operators $\A$ and $\B^*$, which is not restrictive to consider closed. In this way, Assumption\til\ref{ass} is verified by the Closed Graph Theorem. By Remark\til\ref{rem:BS}, since $\dom(\B^*)=C_0^\infty(\R^n;\C^{N})$, then $K_z = \overline{\A(H_0-z)^{-1}\B^*}$. 
Therefore, even in the general case (i.e. removing the assumption of convenience that $\V$ is bounded) we need to study just $\n{\A(H_0-z)^{-1}\B^*}_{\H\to\H}$.

Recall the identity \eqref{Kz-1}. Exploiting the Rigidity Assumptions and setting for simplicity $k^2 \equiv k^2(z) := z^2-m^2$, \eqref{Kz-1} becomes
\begin{align*}
\A (\D_m -z)^{-1} \B^*
=
(m A \alpha_{n+1} B^* + zAB^*) a R_0(k^2) \overline{b}.
\end{align*}
In particular, assume that RA($\iota$) hold, for fixed $\iota\in\{i,ii,iii,iv\}$. Then, since $|A|=|B|=1$, we get 
\begin{equation*}
	\n{A (m\alpha_{n+1} + z I_N)B^*}_{L^\infty} 
	\le
	\varkappa
\end{equation*}
where
\begin{equation*}
	\varkappa \equiv \varkappa(z)
	:=
	\left\{
	\begin{aligned}
	& |k(z)| \left\lvert \frac{z+m}{z-m} \right\rvert^{\frac{\sgn\Re z}{2}}
	&&\text{if $\iota=i$ and $m>0$,}
	\\
	& m
	&&\text{if $\iota=ii$ and $m>0$,}
	\\
	& |z|
	&&\text{if $\iota=iii$, or $\iota=i$ and $m=0$,}
	\\
	& 0
	&&\text{if $\iota=iv$, or $\iota=ii$ and $m=0$.}
	\end{aligned}
	\right.
\end{equation*}
By H\"older's inequality,
\begin{equation*}
\begin{split}
\n{\A (\D_m -z)^{-1} \B^* \phi}_{L^2}
\le
\varkappa
\n{a}_{L^{\frac{2q}{q-2}}}
\n{b}_{L^{\frac{2p}{2-p}}} 
\n{R_0(k^2)}_{L^p \to L^q}
\n{\phi}_{L^2}
\end{split}
\end{equation*}
and so, recalling that $|a|=|b|=|v|^{1/2}$, setting $q=p'$ and $1/r=1/p-1/q$, we get
\begin{equation*}
	\n{\A (\D_m -z)^{-1} \B^* \phi}_{L^2}
	\le
	\varkappa
	\n{v}_{L^r}
	\n{R_0(k^2)}_{L^p \to L^q}
	\n{\phi}_{L^2}.
\end{equation*}
Similarly one infers also
\begin{align*}
\n{\A (\D_m -z)^{-1} \B^*}_{L^2 \to L^2}
&\le 
\varkappa
\n{v}_{L_\rho^{r} L_\theta^\infty}
\n{R_0(k^2)}_{L_\rho^p L_\theta^2 \to L_\rho^q L_\theta^2}
\\
\n{\A (\D_m -z)^{-1} \B^*}_{L^2 \to L^2}
&\le 
\varkappa
\n{v}_{L_\rho^{r,1} L_\theta^\infty}
\n{R_0(k^2)}_{L_\rho^{p,1} L_\theta^2 \to L_\rho^{q,\infty} L_\theta^2}.
\end{align*}
%

From Lemmata\til\ref{resest-T1},\til\ref{resest-T3} and\til\ref{resest-T2} on the conjugate line (hence on the segments $[A_0,B_0]$, $(B_0,C]$ and $(C,E_0]$ respectively), if $n=1$ we get
\begin{align}\label{eq:bs-est-1}
\n{\A (\D_{m} -z)^{-1} \B^*}_{L^2 \to L^2}
\le
\frac{\varkappa}{2} |k|^{-1}
\n{v}_{L^{1}}
\end{align}
whereas, if $n\ge2$, we have
\begin{align}
\label{eq:bs-est}
\n{\A (\D_m -z)^{-1} \B^*}_{L^2 \to L^2}
&\lesssim
\varkappa |k|^{-2+\frac{n}{r}}
\left\{
\begin{aligned}
&
\n{v}_{L^{r}}
&&\text{if $r\in\left[\frac{n}{2},\frac{n+1}{2}\right]$ and $r>1$,}
\\
&
\n{v}_{L_\rho^{r} L_\theta^\infty}
&&\text{if $r\in\left(\frac{n+1}{2},n\right)$,}
\\
&
\n{v}_{L_\rho^{n,1} L_\theta^\infty}
&&\text{if $r=n$,}
\end{aligned}
\right.
\\
\label{eq:bs-est-dist}
\n{\A (\D_m -z)^{-1} \B^*}_{L^2 \to L^2}
&\lesssim
\varkappa 
\frac{|k|^{-\frac{1}{r}}}{\dist(k^2,[0,\infty))^{1-\frac{n+1}{2r}}}
\n{v}_{L^{r}}
\quad
\text{if $r\in\left(\frac{n+1}{2}, \infty \right]$.}
\end{align}
%

In short, we have found inequalities of the type
\begin{equation*}
	\n{\A (\D_m -z)^{-1} \B^*}_{L^2 \to L^2}
	\le C
	\kappa(z) \n{v}
\end{equation*}
for a suitable norm $\n{\cdot}$ of $v$, a positive constant $C$ independent on $z$ and where the function $\kappa$ is either
\begin{equation*}
	\kappa(z) = \varkappa(z) |k(z)|^{-2+\frac{n}{r}}
	\quad\text{or}\quad
	\kappa(z) = \varkappa(z) \frac{|k(z)|^{-1/r}}{\dist(z^2-m^2,[0,\infty))^{1-\frac{n+1}{2r}}}.
\end{equation*} 
Applying the Birman-Schwinger principle and proving our results is now straightforward and easy, maybe just a bit dazzling due to the fauna of cases. According to the hypothesis assumed in the statements of each of our theorems, observe that the region $\mathcal{S}$ described by 
$$\mathcal{S}=\{z\in\C \colon 1 \le C \kappa(z) \n{v} \}$$
in any case covers all the region $\rho(\D_m) = \C\setminus\{\zeta\in\R \colon |\zeta|\ge m\}$. Ergo we can always fix a complex number $z_0 \in \rho(\D_m)$ outside $\mathcal{S}$ satisfying $C K(z_0) \n{v} < 1$, namely Assumption\til\ref{ass2'} is verified (e.g. one can take $z_0=iy_0$, for $y_0 \in \R$ sufficiently large). By Lemma\til\ref{lem:BS} we can deduce that the point spectrum of the perturbed operator $\D_{m,\V}$ is confined in $\mathcal{S}$. If in particular $\kappa(z)$ is a nonnegative constant smaller than $1$ (even $0$, in which case the Birman-Schwinger operator is identically zero), we can exploit Lemma\til\ref{lem:BS2} obtaining that $\sigma(\D_{m,\V})=\sigma_c(\D_{m,\V})=\sigma_c(\D_{m})=(-\infty,-m]\cup[m,\infty)$ and in particular $\sigma_p(\D_{m,\V})=\varnothing$.

In the case of RA(ii) with $m>0$, we have
$\varkappa \equiv 1$ and it is immediate, from the Birman-Schwinger principle and all the above estimates for $\A (\D_{m} -z)^{-1} \B^*$, to conclude the proofs for Theorems\til\ref{thm:RAii-n1}, \ref{thm:RAii} and \ref{thm:RAii-dist}. When we consider RA(ii) with $m=0$ or instead RA(iv), then $\varkappa\equiv0$ and hence the Birman-Schwinger operator is identically zero, implying the stability of the spectrum stated in Theorem\til\ref{thm:RAii-iv}.

Now consider the case of RA(i) and $m>0$. Therefore $\varkappa(z) = |k(z)| \left\lvert \frac{z+m}{z-m} \right\rvert^{\sgn\Re z/2}$ and hence
$\kappa(z)$ is either of the form
\begin{equation}\label{eq:K1}
\kappa(z) = |k(z)|^{-1+\frac{n}{r}}
\left\lvert \frac{z+m}{z-m} \right\rvert^{\sgn\Re z/2}
\end{equation}
or of the form
\begin{equation}\label{eq:K2}
\kappa(z) =  \frac{|k(z)|^{1-1/r}}{\dist(z^2-m^2,[0,\infty))^{1-\frac{n+1}{2r}}} 
\left\lvert \frac{z+m}{z-m} \right\rvert^{\sgn\Re z/2}
\end{equation} 
We are interested in localizing the eigenvalues in compact regions, or at least in neighborhood $\mathcal{N}$ of the continuous spectrum of $\D_{m}$ such that $\mathcal{N} \cap \{z\in\C \colon \Re z = x_0 \}$ is compact for any fixed $x_0\in\R$. At this aim one should ask that $\kappa(z)$ is uniformly bounded as $|z|\to \infty$ in the first case, and that $K(x_0+\Im z)$ is uniformly bounded as $|\Im z| \to \infty$ in the second case.
It is easy to check that if $\kappa(z)$ is like in \eqref{eq:K1}, then
\begin{equation*}
	\kappa(z) \sim |z|^{-1+n/r}
	\quad\text{as $|z|\to\infty$}
\end{equation*}
whereas if $\kappa(z)$ is like in \eqref{eq:K2}, then
\begin{equation*}
	\kappa(\Re z+i\Im z) \sim
	|\Im z|^{-1+\frac{n}{r}} 
	\quad\text{as $|\Im z|\to\infty$}
\end{equation*}
for fixed $\Re z \in\R$. 
%
%
In both cases, we should ask $r \ge n$ to get an interesting (in the sense specified above) localization for the eigenvalues.
The same argument holds in the case of RA(i) and $m=0$, or in the case of RA(iii), namely when $\varkappa(z)=|z|$.
For this reason, to get Theorems\til\ref{thm:RAiii}--\ref{thm:RAi-dist} we only employ the estimates \eqref{eq:bs-est-1}, \eqref{eq:bs-est} for $r=n$ and \eqref{eq:bs-est-dist} for $\gamma:=r-n/2 \ge n/2$.

In particular, \eqref{eq:bs-est-1} and \eqref{eq:bs-est} for $r=n$ imply Theorem\til\ref{thm:RAi} when $\varkappa(z) = |k(z)| \left\lvert \frac{z+m}{z-m} \right\rvert^{\sgn\Re z/2}$, taking in account that
$\frac{C_0}{\n{v}} \le \left\lvert \frac{z+m}{z-m} \right\rvert^{\sgn\Re z/2}$ is equivalent to
\begin{equation*}
	\left(|\Re z| - m \frac{C_0^4+\n{v}^4}{C_0^4-\n{v}^4}\right)^2 + (\Im z)^2 
	\le
	\left(m\frac{2C_0^2 \n{v}^2}{C_0^4-\n{v}^4}\right)^2
\end{equation*}
if $\n{v}< C_0$. In the same case, \eqref{eq:bs-est-dist} implies Theorem\til\ref{thm:RAi-dist}.
When instead $\varkappa(z)=|z|$ and $m=0$, noting that $\varkappa |k|^{-1} \equiv 1$, thanks to \eqref{eq:bs-est-1} and \eqref{eq:bs-est} for $r=n$ we can prove the massless cases in Theorems\til\ref{thm:RAiii} and\til\ref{thm:RAi}. The last inequalities are used to prove also Theorem\til\ref{thm:RAiii}, in the case of RA(iii) and $m>0$. Finally, Theorem\til\ref{thm:RAiii-dist} is proved exploiting \eqref{eq:bs-est-dist} in the case $\varkappa=|z|$.
We conclude noting that in Theorem\til\ref{thm:RAi-dist} and\til\ref{thm:RAiii-dist}, when $\gamma=n/2$, the additional hypothesis $\n{v}_{L^{\gamma+n^2}}^{\gamma+n^2} < C_0$ is necessary, since in this case $K(x_0 + i\Im z) \sim 1$ as $|\Im z|\to \infty$ for fixed $x_0 \in \R$. Hence, if the norm of the potential is not small enough, the condition $\mathcal{N} \cap \{z\in\C \colon \Re z = x_0 \}$ compact would not be satisfied.

	Last but not least, we sketch the proof of Theorem\til\ref{thm:RAi-general}, which is not so different from that of Theorem\til\ref{thm:RAi}. Here we need to use the usual polar decomposition $\V=\mathcal{U}\mathcal{W}=\B^* \A$ with $\A=\sqrt{\mathcal{W}}$ and $\B=\sqrt{\mathcal{W}}\mathcal{U}^*$, described in Section\til\ref{sec:main}. Employing Corollary\til\ref{cor:radial-dirac}, by H\"older's inequality we immediately obtain	
	\begin{equation*}
		\n{\A (\D_{m}-z)^{-1} \B^* \phi}_{L^{2}} 
		\le K 
		\n{\V}_{L_\rho^{n,1} L_\theta^\infty}
		\left[ 1 + \left\lvert \frac{z+m}{z-m} \right\rvert^{\frac{\sgn\Re z}{2}} \right]
		\n{\phi}_{L^{2}}.
	\end{equation*}
	Assumptions I and II are verified as above, and note that in the massive case the inequality 
	\begin{equation*}
		1
		\le K 
		\n{\V}_{L_\rho^{n,1} L_\theta^\infty}
		\left[ 1 + \left\lvert \frac{z+m}{z-m} \right\rvert^{\frac{\sgn\Re z}{2}} \right]
	\end{equation*}
	describes the two disks in the statement of Theorem\til\ref{thm:RAi-general}, letting $C_0=\frac{1}{2K}$. Another application of the Birman-Schwinger principle concludes the proof.


\section{Game of matrices}\label{sec:potential}

The present section is fully dedicated to computations with matrices in order to exhibit some explicit examples of potentials $\V=vV$ such that the matricial part $V$ can be factorized in the product of two matrices $B^*$ and $A$ satisfying the various assumptions stated in Section\til\ref{sec:main}. 

We will prove that in some low dimensions it is not possible to find the required potential, more precisely they do not exists in dimension $n=2,4$ in the case of RA(ii) and RA(iii), and in dimension $n=1,2$ in the case of RA(iv). In all the other case, we will exhibit at least a couple of examples. There is no intent here to be exhaustive in finding the suitable matrices, but rather we want to suggest an idea to build them. At this aim we firstly need to show an explicit representation for the Dirac matrices, and then we need to introduce some special \lq\lq brick'' matrices.

\subsection{The Dirac matrices}

First of all, as anticipated in Remark\til\ref{rem:dirac}, let us explicitly define the Dirac matrices we are going to employ in our calculations, or better, one of their possible representations. At this aim we rely on the recursive construction performed by Kalf and Yamada in the Appendix of \cite{KalfYamada01}.

Let us introduce the Pauli matrices
\begin{equation*}
	\sigma_1 :=
	\begin{pmatrix}
	0 & 1 
	\\
	1 & 0
	\end{pmatrix}
	,
	\quad
	\sigma_2 :=
	\begin{pmatrix}
	0 & -i
	\\
	i & 0
	\end{pmatrix}
	,
	\quad
	\sigma_3 :=
	\begin{pmatrix}
	1 & 0
	\\
	0 & -1
	\end{pmatrix}.
\end{equation*}
Moreover, let us define the Kronecker product between two matrices $A=(a_{ij}) \in \C^{r_1\times c_1}$and $B=(b_{ij}) \in \C^{r_2 \times c_2}$, with $r_1,c_1,r_2,c_2\in\N$, as follows:
\begin{equation*}
	A \tensor B :=
	\begin{pmatrix}
	a_{11} B & \cdots & a_{1n} B
	\\
	\vdots & \ddots & \vdots
	\\
	a_{n1} B & \cdots & a_{nn} B
	\end{pmatrix}
	\in\C^{r_1 r_2 \times c_1 c_2}.
\end{equation*}
Recall that the Kronecker product satisfies, among others, the associative property and the mixed-product property, viz.
\begin{gather*}
	A_1 \tensor (A_2 \tensor A_3) = (A_1 \tensor A_2) \tensor A_3 = A_1 \tensor A_2 \tensor A_3
	\\
	(A_1 \tensor B_1)(A_2 \tensor B_2) = (A_1 A_2) \tensor (B_1 B_2).
\end{gather*}

The Dirac matrices in low dimensions can be chosen to be the Pauli matrices, namely for $n=1$ we set
\begin{equation*}
	\alpha_1^{(1)} := \sigma_1,
	\quad
	\alpha_{2}^{(1)} := \sigma_3,
\end{equation*}
and for $n=2$
\begin{equation*}
\alpha_1^{(2)} := \sigma_1,
\quad
\alpha_2^{(2)} := \sigma_2,
\quad
\alpha_{3}^{(2)} := \sigma_3.
\end{equation*}
The apex $(n)$ stands for the dimension; we will omit it when there is no possibility of confusion.
Let us start the recursion, after recalling that we defined $N=2^{\lceil n/2 \rceil}$: 
\begin{enumerate}[label=(\roman*)]
	\item 
	if $n\ge3$ is odd, we use the matrices $\alpha_1^{(n-1)}, \dots, \alpha_{n+1}^{(n-1)}$ known from the dimension $n-1$ to construct
	\begin{equation*}\label{rec-oddDmat}
	\alpha_k^{(n)} :=
	\sigma_1 \tensor \alpha_k^{(n-1)}
	,\qquad
	\alpha_{n+1}^{(n)} :=
	\sigma_3 \tensor I_{N/2}
	\end{equation*}
	%
	%
	for $k\in\{1,\dots,n\}$;
	
	\item if the dimension $n\ge4$ is even, we define
	%
	%
	\begin{equation*}\label{rec-evenDmat}
	\alpha_1^{(n)} :=
	\sigma_1 \tensor I_{N/2}
	,\qquad
	\alpha_{k+1}^{(n)} :=
	\sigma_2 \tensor \alpha_k^{(n-2)}
	,\qquad
	\alpha_{n+1}^{(n)} :=
	\sigma_3 \tensor I_{N/2}
	\end{equation*}
	for $k\in\{1,\dots,n-1\}$.	
\end{enumerate}
In any dimension $n\ge1$, the Dirac matrices $\alpha_1,\dots,\alpha_{n+1}$ just defined are Hermitian, satisfy \eqref{clifford} and have the structure
\begin{equation*}\label{eq:beta}
\alpha_k =
	\begin{pmatrix}
	{0} & \beta_k
	\\
	\beta_k^* & {0}
	\end{pmatrix}
	,
	\quad
\alpha_{n+1} =
	\begin{pmatrix}
	I_{N/2} & {0}
	\\
	{0} & -I_{N/2}
	\end{pmatrix}
\end{equation*}
for $k \in\{1,\dots,n\}$, where the matrices $\beta_k \in \C^{N/2 \times N/2}$ satisfy
\begin{equation*}
	\beta_k \beta_j^* + \beta_j \beta_k^* = 2 \delta_k^j I_{N/2}
\end{equation*}
and are Hermitian if $n$ is odd.

\begin{remark}
	Not only in dimension $n=1,2$, but also in dimension $n=3$, the above representation for the Dirac matrices coincides with the classical one:
	\begin{align*}
		\alpha_1^{(3)} = \sigma_1 \tensor \sigma_1
		&=
		\begin{pmatrix}
		0 & 0 & 0 & 1
		\\
		0 & 0 & 1 & 0
		\\
		0 & 1 & 0 & 0
		\\
		1 & 0 & 0 & 0
		\end{pmatrix},
		&
		\alpha_2^{(3)} = \sigma_1 \tensor \sigma_2
		&=
		\begin{pmatrix}
		0 & 0 & 0 & -i
		\\
		0 & 0 & i & 0
		\\
		0 & -i & 0 & 0
		\\
		i & 0 & 0 & 0
		\end{pmatrix},
		\\
		\alpha_3^{(3)} = \sigma_1 \tensor \sigma_3
		&=
		\begin{pmatrix}
		0 & 0 & 1 & 0
		\\
		0 & 0 & 0 & -1
		\\
		1 & 0 & 0 & 0
		\\
		0 & -1 & 0 & 0
		\end{pmatrix},
		&
		\alpha_4^{(3)} = \sigma_3 \tensor I_2
		&=
		\begin{pmatrix}
		1 & 0 & 0 & 0
		\\
		0 & 1 & 0 & 0
		\\
		0 & 0 & -1 & 0
		\\
		0 & 0 & 0 & -1
		\end{pmatrix}.
	\end{align*}
\end{remark}

\begin{remark}\label{rem:not-restrictive}
	If $\{ \alpha_1, \dots, \alpha_{n+1}\}$ and $\{ \widetilde{\alpha}_1, \dots, \widetilde{\alpha}_{n+1}\}$ are a pair of sets of Dirac matrices, then there exists a unitary matrix $U \in \C^{N\times N}$ such that $\widetilde{\alpha}_k = U \alpha_k U^{-1}$ or $\widetilde{\alpha}_k = - U \alpha_k U^{-1}$, for $k\in\{1,\dots,n+1\}$. If $n$ is odd we always fall in the first case; if $n$ is even and we are in the second case,
	set $\widetilde{U}=U \prod_{k=1}^{n} \alpha_k$, then
	\begin{equation*}
	\widetilde{\D}_{m} = -i \sum_{j=1}^{n} \widetilde{\alpha}_k \partial_k + m \widetilde{\alpha}_{n+1} 
	= \widetilde{U}
	\left[-i \sum_{j=1}^{n} {\alpha}_k \partial_k - m {\alpha}_{n+1}\right] \widetilde{U}^{-1}
	=
	\widetilde{U}
	{\D}_{-m} 
	\widetilde{U}^{-1}
	.
	\end{equation*}
	Therefore, considering the perturbed operator $\widetilde{\D}_{m,\widetilde{\V}}$, in odd dimension it is unitarily equivalent to ${\D}_{m,\V}$ with $\widetilde{\V}= \widetilde{U} {\V} \widetilde{U}^{-1}$, whereas in even dimension it is unitarily equivalent to either $\D_{m,\V}$ or $\D_{-m,\V}$. 
	
	In the case of the current paper, noting that all the results in Section\til\ref{sec:main} are symmetric respect to the imaginary axis (namely they are not effected replacing $m$ with $-m$ in the definition of the Dirac operator), it becomes evident that the choice of a particular representation for the Dirac matrices is not restrictive at all. 
\end{remark}


The above recursive definition for the matrices may appear too much implicit, but we can go further exploding the representation. Let us define the \lq\lq Kronecker exponentiation''
\begin{align*}
	M^{\tensor 0} &= 1
	\\
	M^{\tensor k} &= \underbrace{M \tensor \cdots \tensor M}_\text{$k$ times}
\end{align*}
for any complex matrix $M$ and for any $k\in\N$, imposing the natural identification between $1\in\C$ and the matrix $(1) \in \C^{1\times1}$.
Therefore, one can explicitly write the Dirac matrices in even dimension $n\ge2$ as 
\begin{equation}\label{evenDmat}
	\alpha_k =
	\left\{
		\begin{aligned}
			&\sigma_2^{\tensor k-1} \tensor \sigma_1 \tensor I_{2}^{\tensor n/2-k}
			&&\text{for $k\in\left\{1,\dots,\frac{n}{2}\right\}$}
			\\
			&\sigma_2^{\tensor n/2} 
			&&\text{for $k=\frac{n}{2}+1$}
			\\
			&\sigma_2^{\tensor n+1-k} \tensor \sigma_3 \tensor I_{2}^{\tensor k-n/2-2}
			&&\text{for $k\in\left\{\frac{n}{2}+2,\dots,n+1\right\}$}
		\end{aligned}
	\right.
\end{equation}
and in odd dimension $n\ge3$ as
\begin{equation*}\label{oddDmat}
	\alpha_k =
	\left\{
	\begin{aligned}
	&\sigma_1\tensor\sigma_2^{\tensor k-1} \tensor \sigma_1 \tensor I_{2}^{\tensor (n-1)/2-k}
	&&\text{for $k\in\left\{1,\dots,\frac{n-1}{2}\right\}$}
	\\
	&\sigma_1\tensor\sigma_2^{\tensor (n-1)/2} 
	&&\text{for $k=\frac{n-1}{2}+1$}
	\\
	&\sigma_1\tensor\sigma_2^{\tensor n-k} \tensor \sigma_3 \tensor I_{2}^{\tensor k-(n-1)/2-2}
	&&\text{for $k\in\left\{\frac{n-1}{2}+2,\dots,n\right\}$}
	\\
	&\sigma_3\tensor I_{2^{}}^{\tensor (n-1)/2}
	&&\text{for $k=n+1$.}
	\end{aligned}
	\right.
\end{equation*}
The odd dimensional case follow easily from the recursive definition and from the explicit definition \eqref{evenDmat} of the Dirac matrices in the even dimensional case; the latter can be easily verified by induction, and we omit the proof.

For later use, we collect in the following lemma a recursive formula which connects the Dirac matrices associated to two different dimensions.

\begin{lemma}\label{lemma:recursive_dirac}
	Let $n,m\in\N$ such that $2 \le m \le n$ and $n-m$ is even. Thus the following identity hold:
	\begin{equation*}
		\alpha_k^{(n)}
		=
		\left\{
		\begin{aligned}
		&\alpha_k^{(m)} \tensor I_2^{\tensor (n-m)/2}
		&&\text{for $k\in\left\{ 1,\dots, \left\lfloor\frac{m}{2}\right\rfloor \right\}$}
		\\
		&\alpha_{\left\lfloor m/2 \right\rfloor+1}^{(m)} \tensor \alpha_{k- \left\lfloor m/2 \right\rfloor}^{(n-m)}
		&&\text{for $k\in\left\{ \left\lfloor\frac{m}{2}\right\rfloor+1,\dots, n-m+\left\lfloor \frac{m}{2} \right\rfloor +1 \right\}$}
		\\
		&\alpha_{k-(n-m)}^{(m)} \tensor I_2^{\tensor (n-m)/2}
		&&\text{for $k\in\left\{ n-m+\left\lfloor \frac{m}{2} \right\rfloor +2,\dots, n+1 \right\}$}
		\end{aligned}
		\right.
	\end{equation*}
	where $\lfloor\cdot\rfloor$is the floor function.
\end{lemma}

\begin{proof}
If $n,m$ are both even, we want to prove
\begin{equation}\label{even-recursive2}
	\alpha_k^{(n)}
	=
	\left\{
		\begin{aligned}
			&\alpha_k^{(m)} \tensor I_2^{\tensor (n-m)/2}
			&&\text{for $k\in\left\{ 1,\dots, \frac{m}{2} \right\}$}
			\\
			&\alpha_{m/2+1}^{(m)} \tensor \alpha_{k-m/2}^{(n-m)}
			&&\text{for $k\in\left\{ \frac{m}{2}+1,\dots, n-\frac{m}{2}+1 \right\}$}
			\\
			&\alpha_{k-(n-m)}^{(m)} \tensor I_2^{\tensor (n-m)/2}
			&&\text{for $k\in\left\{ n-\frac{m}{2}+2,\dots, n+1 \right\}$.}
		\end{aligned}
	\right.
\end{equation}
But, from \eqref{evenDmat} and setting for simplicity $j:=k-m/2$, $h:=k-(n-m)$ for any $k\in\N$, we immediately have
\begin{equation*}
	\alpha_k^{(n)} =
	\left\{
	\begin{aligned}
		&\sigma_2^{\tensor k-1} \tensor \sigma_1 \tensor I_2^{\tensor m/2-k} \tensor I_2^{\tensor (n-m)/2}
		&&\text{for $k\in\{1,\dots,m/2\}$}
		\\
		&\sigma_2^{\tensor m/2} \tensor \sigma_2^{j-1} \tensor \sigma_1 \tensor I_2^{\tensor (n-m)/2-j}
		&&\text{for $k\in\{m/2+1,\dots,n/2\}$}
		\\
		& \sigma_2^{\tensor m/2} \tensor \sigma_2^{\tensor (n-m)/2}
		&&\text{for $k=n/2+1$}
		\\
		&\sigma_2^{\tensor m/2} \tensor \sigma_2^{n-m+1-j} \tensor \sigma_3 \tensor I_2^{\tensor j-(n-m)/2-2}
		&&\text{for $k\in\{n/2+2,\dots,n-m/2+1\}$}
		\\
		&\sigma_2^{\tensor m+1-h} \tensor \sigma_3 \tensor I_2^{\tensor h-m/2-2} \tensor I_2^{\tensor (n-m)/2}
		&&\text{for $k\in\{n-m/2+2,\dots,n+1\}$}
	\end{aligned}
	\right.
\end{equation*}
from which our assertion is evident. 

If $n,m$ are both odd, exploiting \eqref{even-recursive2} it follows that 
\begin{equation*}\label{odd-recursive2}
\begin{split}
	\alpha_k^{(n)}
	&=
	%
	\left\{
		\begin{aligned}
		&\sigma_1 \tensor \alpha_k^{(n-1)}
		&&\text{for $k\in\{1,\dots,n\}$}
		\\
		&\sigma_3 \tensor I_2^{\tensor n/2-1}
		&&\text{for $k=n+1$}
		\end{aligned}
	\right.
	\\
	&=
	\left\{
	\begin{aligned}
		&\sigma_1 \tensor\alpha_k^{(m-1)} \tensor I_2^{\tensor (n-m)/2}
		&&\text{for $k\in\left\{ 1,\dots, \frac{m-1}{2} \right\}$}
		\\
		&\sigma_1 \tensor\alpha_{(m-1)/2+1}^{(m-1)} \tensor \alpha_{k-(m-1)/2}^{(n-m)}
		&&\text{for $k\in\left\{ \frac{m-1}{2}+1,\dots, n-\frac{m-1}{2} \right\}$}
		\\
		&\sigma_1 \tensor\alpha_{k-(n-m)}^{(m-1)} \tensor I_2^{\tensor (n-m)/2}
		&&\text{for $k\in\left\{ n-\frac{m-1}{2}+1,\dots, n \right\}$}
		\\
		&\sigma_3 \tensor I_2^{\tensor n/2-1}
		&&\text{for $k=n+1$}
	\end{aligned}
	\right.
	\\
	&=
	%
	\left\{
	\begin{aligned}
		&\alpha_k^{(m)} \tensor I_2^{\tensor (n-m)/2}
		&&\text{for $k\in\left\{ 1,\dots, \frac{m-1}{2} \right\}$}
		\\
		&\alpha_{(m-1)/2+1}^{(m)} \tensor \alpha_{k-(m-1)/2}^{(n-m)}
		&&\text{for $k\in\left\{ \frac{m-1}{2}+1,\dots, n-\frac{m-1}{2} \right\}$}
		\\
		&\alpha_{k-(n-m)}^{(m)} \tensor I_2^{\tensor (n-m)/2}
		&&\text{for $k\in\left\{ n-\frac{m-1}{2}+1,\dots, n+1 \right\}$}
	\end{aligned}
	\right.
\end{split}
\end{equation*}
which concludes the proof.
\end{proof}

To conclude this subsection on the Dirac matrices, it seems interesting to use noting the following relation about their product, even if we are not going to exploit it.

\begin{lemma}
We have that
\begin{equation}\label{def:atilde}
\widetilde{\alpha}
:= 
(-i)^{\lfloor\frac{n}{2}\rfloor}\prod_{k=1}^{n+1}\alpha_k
=
\begin{cases}
-i\sigma_2 \tensor I_{N/2} &\text{if $n$ is odd,}
\\
I_{N} &\text{if $n$ is even,}
\end{cases}
\end{equation}
and in particular
\begin{align*}
	\widetilde{\alpha}^2 = (-1)^n I_N,
	\qquad
	\widetilde{\alpha}^* = (-1)^n \widetilde{\alpha},
	\qquad
	\alpha_k \widetilde{\alpha} = (-1)^{n}  \widetilde{\alpha} \alpha_k
\end{align*}
for $k\in\{1,\dots,n+1\}$.
\end{lemma}

\begin{proof}
	The three properties follows obviously from the anticommutation relations \eqref{clifford}, so we need just to prove the second equality in \eqref{def:atilde}. Suppose firstly that $n$ is even. Then the identity follows by inductive argument.
	If $n=2$, it is directly verified that
	\begin{equation*}
	-i \alpha_1^{(2)} \alpha_2^{(2)} \alpha_3^{(2)} 
	= -i \sigma_1  \sigma_2  \sigma_3 		
	= I_2.
	\end{equation*}
	Fix now $n\ge4$ even and suppose that
	\begin{equation*}
	(-i)^{n/2-1} \prod_{k=1}^{n-1} \alpha_k^{(n-2)} = I_{N/2}.
	\end{equation*}
	Exploiting the definitions of the Dirac matrices and the mixed-product property of the Kronecker product, we get
	\begin{equation*}
	\begin{split}
	(-i)^{n/2} \prod_{k=1}^{n+1} \alpha_k^{(n)}
	&=
	(-i)^{n/2} 
	\left[ \sigma_1 \tensor I_{N/2} \right]
	\left[ \prod_{k=1}^{n-1} \sigma_2 \tensor \alpha_k^{(n-2)}\right]
	\left[ \sigma_3 \tensor I_{N/2} \right]
	\\
	&=
	(-i)^{n/2} 
	\left[ \sigma_1 \tensor I_{N/2} \right]
	\left[ \sigma_2^{n-1} \tensor \prod_{k=1}^{n-1}  \alpha_k^{(n-2)}\right]
	\left[ \sigma_3 \tensor I_{N/2} \right]
	\\
	&=
	-i 
	\left[ \sigma_1 \tensor I_{N/2} \right]
	\left[ \sigma_2 \tensor I_{N/2} \right]
	\left[ \sigma_3 \tensor I_{N/2} \right]
	\\
	&=
	-i \sigma_1  \sigma_2  \sigma_3  \tensor I_{N/2}
	\\
	&=
	I_N
	\end{split}
	\end{equation*}
	Finally, let $n\ge1$ be odd. If $n=1$, then it is trivially checked that $\alpha_1^{(1)}\alpha_2^{(1)}=\sigma_1\sigma_3=-i\sigma_2$. If $n\ge3$, then 
	\begin{equation*}
	(-i)^{\frac{n-1}{2}} \prod_{k=1}^n \alpha_k^{(n)} 
	=
	(-i)^{\frac{n-1}{2}} \prod_{k=1}^n \sigma_1 \tensor \alpha_k^{(n-1)}
	=
	(-i)^{\frac{n-1}{2}} \sigma_1^n \tensor \prod_{k=1}^n \alpha_k^{(n-1)}
	=
	\sigma_1 \tensor I_{N/2}
	\end{equation*}
	and hence
	\begin{equation*}
		(-i)^{\frac{n-1}{2}} \prod_{k=1}^{n+1} \alpha_k^{(n)} 
		=
		[\sigma_1 \tensor I_{N/2}][\sigma_3 \tensor I_{N/2}]
		=
		-i\sigma_2 \tensor I_{N/2}
	\end{equation*}
	concluding the proof of the identity.
\end{proof}


\subsection{The brick matrices}

Before to proceed with the construction of the examples for the potentials, we need to point our attention on some peculiar $2\times2$ matrices. 
We want to find $\rho^k, \tau^k \in \C^{2\times2}$ satisfying the conditions 
\begin{gather*}
	\rho^k\sigma_{1}(\tau^k)^* =0= \rho^k\sigma_{k}(\tau^k)^*
	\\
	\rho^k\sigma_{h}(\tau^k)^* 
	\neq 0 \neq
	(\tau^k)^*\rho^k
\end{gather*}
for fixed $k\in\{0,2,3\}$ and any $h\in\{0,2,3\}\setminus\{k\}$, where we define for simplicity $\sigma_0:=I_2$ 
Moreover, let us ask also $|\rho^k|=|\tau^k|=1$, where $|\cdot|$ is the matricial $2$-norm, a.k.a. the spectral norm.
It is quite simple to find a couple of such matrices for any $k\in\{0,2,3\}$, properly combining the Pauli matrices. 

In the case $k=0$, we can consider
\begin{equation*}
\rho^0_+ = \frac{ \sigma_2 + i \sigma_3}{2} = \tau^0_-,
\quad
\tau^0_+ = \frac{ I_2 + \sigma_1}{2} = \rho^0_-,
\end{equation*}
from which, using the anticommutation relations \eqref{clifford}, it easy to check
\begin{gather*}
\rho^0_\pm \sigma_0 (\tau_\pm^0)^* 
=0=
\rho^0_\pm \sigma_1 (\tau_\pm^0)^*
\\
\rho^0_\pm \sigma_2 (\tau_\pm^0)^*
= 
\frac{I_2 + \sigma_1}{2}
=
\mp i \rho^0_\pm \sigma_3 (\tau_\pm^0)^* 
\\
(\tau_\pm^0)^* \rho^0_\pm = \frac{\sigma_2 \pm  i\sigma_3}{2}.
\end{gather*}
In the case $k=2$, we can set
\begin{equation*}
	\rho^2_\pm = \frac{\sigma_1 \mp i \sigma_2}{2} =
	\tau^2_\pm,
\end{equation*}
and thus
\begin{gather*}
	\rho^2_\pm \sigma_1 (\tau_\pm^2)^* 
	=0=
	\rho^2_\pm \sigma_2 (\tau_\pm^2)^*
	\\
	\rho^2_\pm \sigma_0 (\tau_\pm^2)^*
	= 
	\frac{I_2 \mp \sigma_3}{2}
	=
	\pm \rho^2_\pm \sigma_3 (\tau_\pm^2)^* 
	\\
	(\tau_\pm^2)^* \rho^2_\pm = \frac{I_2 \pm \sigma_3}{2}.
\end{gather*}
Finally, in the case $k=3$, we can consider
\begin{equation*}
\rho^3_\pm = \frac{I_2 \pm \sigma_2}{2} = \tau^3_\pm
\end{equation*}
and hence
\begin{gather*}
	\rho^3_\pm \sigma_1 (\tau_\pm^3)^* 
	=0=
	\rho^3_\pm \sigma_3 (\tau_\pm^3)^*
	\\
	\rho^3_\pm \sigma_0 (\tau_\pm^3)^*
	= 
	\frac{I_2 \pm \sigma_2}{2}
	=
	\pm \rho^3_\pm \sigma_2 (\tau_\pm^3)^* 
	\\
	(\tau_\pm^3)^* \rho^3_\pm = \frac{I_2 \pm \sigma_2}{2}.
\end{gather*}

The couple of matrices we found for each of the three cases are not the only solutions satisfying the required set of conditions, but for our purposes are enough.

Now, we want to find matrices $A, B \in \C^{\N\times N}$ such that
\begin{gather*}
A \alpha_k B^* = 0
\\
V=B^*A \neq 0
\end{gather*}
for $k \in\{1,\dots,n\}$. In addition, we will also impose, or not, that $AB^*$ and $A\alpha_{n+1} B^*$ are null matrices.

\subsection{The odd-dimensional case}

Let us start with the $1$-dimensional case, for which we have basically already found the admissible matricial part for the potentials, thanks to the brick matrices found in the previous subsection.

In fact, to satisfy RA(i) we need to find $V=B^*A \neq 0$ such that $A\sigma_1 B^*=0$, $AB^*\neq0$, $A\sigma_3 B^*\neq0$, thus we can choose $A_\pm = \rho^2_\pm, B_\pm = \tau^2_\pm$, obtaining the couple of examples
\begin{align*}
	V_\pm
	=
	\frac{1}{2}
	\begin{pmatrix}
	1 \pm 1 & 0
	\\
	0 & 1 \mp 1
	\end{pmatrix}
	=
	\left[
	\frac{1}{2}
	\begin{pmatrix}
		0 & 1 \mp 1
		\\
		1 \pm 1 & 0 
	\end{pmatrix}
	\right]^*
	\left[
	\frac{1}{2}
	\begin{pmatrix}
		0 & 1 \mp 1
		\\
		1 \pm 1 & 0 
	\end{pmatrix}
	\right]
	=
	B_\pm^* A_\pm.
\end{align*}
Similarly we can proceed for the case of RA(ii) and RA(iii), in which case we use $\rho_\pm^0, \tau_\pm^0$ and $\rho_\pm^3, \tau_\pm^3$ respectively, viz. for RA(ii) we have the couple of examples 
\begin{align*}
V_\pm
&=
\frac{i}{2}
\begin{pmatrix}
\pm 1 & -1
\\
1 & \mp 1
\end{pmatrix}
=
\left[
\frac{1}{2}
\begin{pmatrix}
1 & \pm 1
\\
1 & \pm 1
\end{pmatrix}
\right]^*
\left[
\frac{i}{2}
\begin{pmatrix}
1 & \mp 1
\\
1 & \mp 1
\end{pmatrix}
\right]
=
B_\pm^* A_\pm,
\end{align*}
while for RA(iii) we have the couple of examples
\begin{align*}
	V_\pm
	&=
	\frac{1}{2}
	\begin{pmatrix}
	1 & \mp i
	\\
	\pm i & 1
	\end{pmatrix}
	=
	\left[
	\frac{1}{2}
	\begin{pmatrix}
	1 & \mp i
	\\
	\pm i & 1
	\end{pmatrix}
	\right]^*
	\left[
	\frac{i}{2}
	\begin{pmatrix}
	1 & \mp i
	\\
	\pm i & 1
	\end{pmatrix}
	\right]
	=
	B_\pm^* A_\pm.
\end{align*}

This examples can be easily generalized in any odd dimension, taking in account the following

\begin{lemma}\label{lem:rec-pot}
	If $V^{(n-2)} =[B^{(n-2)}]^* A^{(n-2)}$ is an admissible matrix in dimension $n-2$, then an admissible matrix in dimension $n$ is given by
	\begin{equation*}
	V^{(n)} := V^{(n-2)} \tensor I_{2} = \left[B^{(n-2)} \tensor M^{-1} \right]^* \left[A^{(n-2)} \tensor M \right] =: [B^{(n)}]^* A^{(n)}
	\end{equation*}
	for any invertible matrix $M \in \C^{2\times2}$.
\end{lemma}

This assertion is a trivial consequence of Lemma\til\ref{lemma:recursive_dirac}. Thus, in any odd dimension $n\ge 1$, couples of examples satisfying RA(i), RA(ii) and RA(iii) are given respectively by
\begin{equation*}
	V_\pm 
	=
	\frac{1}{2}
	\begin{pmatrix}
	I_{\frac{N}{2}} \pm I_{\frac{N}{2}} & 0
	\\
	0 & I_{\frac{N}{2}} \mp I_{\frac{N}{2}}
	\end{pmatrix}
	,
	\quad
	V_\pm
	=
	\frac{i}{2}
	\begin{pmatrix}
	\pm I_{\frac{N}{2}} & -I_{\frac{N}{2}}
	\\
	I_{\frac{N}{2}} & \mp I_{\frac{N}{2}}
	\end{pmatrix},
	\quad
	V_\pm
	=
	\frac{1}{2}
	\begin{pmatrix}
	I_{\frac{N}{2}} & \mp i I_{\frac{N}{2}}
	\\
	\pm i I_{\frac{N}{2}} & I_{\frac{N}{2}}
	\end{pmatrix}
	.
\end{equation*}

Let us turn now our attention to the case of RA(iv), for which there are no examples of potentials in dimension $n=1$. Indeed, let us fix $A, B \in \C^{2\times 2}$ and let us denote with $a=(a_1, a_2)$ and $b=(b_1, b_2)$ their respective first rows. Since we are imposing 
\begin{equation*}
	A\sigma_1B^*= A\sigma_3 B^* = AB^* =0,
\end{equation*}
in particular we obtain that
\begin{align*}
	a_2 \overline{b_1} + a_1 \overline{b_2} 
	=
	a_1 \overline{b_1} - a_2 \overline{b_2} 
	=
	a_1 \overline{b_1} + a_2 \overline{b_2} 
	=0,
\end{align*}
from which we deduce that if $a\neq0$, then $b=0$, and vice versa if $b\neq0$, then $a=0$. Therefore, one can easily be convinced that there are no solutions such that both $A$ and $B$ are non-trivial.

Let us consider then the $3$-dimensional case. By the definition of the Dirac matrices, we would like to find matrices $A,B$ such that $B^*A \neq 0$ and
\begin{equation*}
	A (\sigma_1 \tensor \sigma_1) B^*
	=
	A (\sigma_1 \tensor \sigma_2) B^*
	=
	A (\sigma_1 \tensor \sigma_3) B^*
	=
	A (\sigma_3 \tensor I_2) B^*
	=
	A (I_2 \tensor I_2) B^*
	=
	0.
\end{equation*}
Anyway, from the properties of our brick matrices and by the mixed-product property of the Kronecker product, it is readily seen that we can choose $A=\rho^k_\pm \tensor \rho^0_\pm$ and $B=\tau^k_\pm \tensor \tau^0_\pm$ for any $k\in\{0,2,3\}$. In fact
\begin{equation*}
	(\rho^k_\pm \tensor \rho^0_\pm)
	(\sigma_1 \tensor \sigma_j)
	(\tau^k_\pm \tensor \tau^0_\pm)^*
	=
	\rho^k_\pm \sigma_1 (\tau^k_\pm)^*
	\tensor 
	\rho^0_\pm \sigma_j (\tau^0_\pm)^*
	=
	0 \tensor \rho^0_\pm \sigma_j (\tau^0_\pm)^*
	=
	0
\end{equation*}
for $j\in\{1,2,3\}$, and
\begin{equation*}
	(\rho^k_\pm \tensor \rho^0_\pm)
	(\sigma_h \tensor I_2)
	(\tau^k_\pm \tensor \tau^0_\pm)^*
	=
	\rho^k_\pm \sigma_h (\tau^k_\pm)^*
	\tensor 
	\rho^0_\pm I_2 (\tau^0_\pm)^*
	=
	\rho^k_\pm \sigma_h (\tau^k_\pm)^* \tensor 0
	=
	0
\end{equation*}
for $h\in\{0,3\}$.
Essentially, we use the fact that the first tensorial factors appearing in the definitions of $A$ and $B$ kill $\sigma_1$, while the second tensorial factors kill $I_2$.
At this point, as above we can extend the $3$-dimensional case to any odd dimension $n\ge5$.

Exempli gratia, letting $k=0$, we have that a couple of examples of matricial part of potentials satisfying RA(iv) for odd dimension $n\ge3$ are given by
\begin{align*}
	V_\pm
	&=
	\frac{1}{4}
	\begin{pmatrix}
	- 1 & \pm 1 & \pm 1 & - 1
	\\
	\mp 1 & 1 & 1 & \mp 1
	\\
	\mp 1 & 1 & 1 & \mp 1
	\\
	- 1 & \pm 1 & \pm 1 & - 1
	\end{pmatrix}
	\tensor
	I_{N/4}
	\\
	&=
	-\frac{1}{4}
	\begin{pmatrix}
	\pm 1 & -1
	\\
	1 & \mp 1
	\end{pmatrix}^{\tensor 2}
	\tensor
	I_{N/4}
	\\
	&=
	\left[
	\frac{1}{4}
	\begin{pmatrix}
	1 & \pm 1
	\\
	1 & \pm 1
	\end{pmatrix}^{\tensor 2}
	\tensor
	I_{N/4}
	\right]^*
	\left[
	-\frac{1}{4}
	\begin{pmatrix}
	1 & \mp 1
	\\
	1 & \mp 1
	\end{pmatrix}^{\tensor 2}
	\tensor
	I_{N/4}
	\right]
	\\
	&=
	\left[
	\frac{1}{4}
	\begin{pmatrix}
	1 & \pm 1 & \pm 1 & 1
	\\
	1 & \pm 1 & \pm 1 & 1
	\\
	1 & \pm 1 & \pm 1 & 1
	\\
	1 & \pm 1 & \pm 1 & 1
	\end{pmatrix}
	\tensor
	I_{N/4}
	\right]^*
	\left[
	\frac{1}{4}
	\begin{pmatrix}
	-1 & \pm 1 & \pm 1 & -1
	\\
	-1 & \pm 1 & \pm 1 & -1
	\\
	-1 & \pm 1 & \pm 1 & -1
	\\
	-1 & \pm 1 & \pm 1 & -1
	\end{pmatrix}
	\tensor
	I_{N/4}
	\right]
	\\
	&=
	B_\pm^* A_\pm.
\end{align*}


\subsection{The even-dimensional case}

We will consider the situation case by case for RA(i)--(iv).

\subsubsection{Case of RA(i)} Between the four cases, this is the only one for which we can find examples of our desired potentials in any dimension. Indeed, let us start from $n=2$, for which a couple of examples can be found immediately exploiting our brick matrices, setting $A=\rho^2_\pm$ and $B=\tau^2_\pm$. Therefore, making use of Lemma\til\ref{lem:rec-pot}, a couple of examples for the matricial part of the potentials satisfying RA(i) for any even dimension $n\ge2$ is given by
\begin{align*}
	V_\pm &= 
	\frac{1}{2} 
	\begin{pmatrix}
	I_{N/2} \pm I_{N/2} & 0
	\\
	0 & I_{N/2} \mp I_{N/2}
	\end{pmatrix}
	\\
	&=
	\left[
	\frac{1}{2}
	\begin{pmatrix}
	0 & I_{N/2} \mp I_{N/2}
	\\
	I_{N/2} \pm I_{N/2} & 0
	\end{pmatrix}
	\right]^*
	\left[
	\frac{1}{2}
	\begin{pmatrix}
	0 & I_{N/2} \mp I_{N/2}
	\\
	I_{N/2} \pm I_{N/2} & 0
	\end{pmatrix}
	\right]
	\\
	&=
	B_\pm^* A_\pm.
\end{align*}


\subsubsection{Case of RA(ii)}
 We can find potentials only for $n\ge6$. Indeed, in dimension $n=2$, the situation is similar to the case of RA(iv) for $n=1$. We are searching matrices $A,B \in \C^{2\times2}$ such that $V=B^*A\neq0$, $A \sigma_3 B^* \neq 0$ and
\begin{equation*}
A \sigma_1 B^* = A \sigma_2 B^* = A B^*=0.
\end{equation*}
Denoting with $a=(a_1,a_2)$ and $b=(b_1,b_2)$ the first rows of respectively $A$ and $B$, from the previous condition we infer
\begin{equation*}
a_2 \overline{b_1} + a_1 \overline{b_2}
=
a_2 \overline{b_1} - a_1 \overline{b_2}
=
a_1 \overline{b_1} + a_2 \overline{b_2}
=
0
\end{equation*}
and therefore $a=0$ if $b\neq0$ and on the contrary $b=0$ if $a\neq0$. Thus, there are no solutions such that $A \neq 0$ and $B\neq0$. 

Analogously, we can repeat the argument for $n=4$. In this case the Dirac matrices are
\begin{equation}\label{eq:dirac4}
\begin{gathered}
\alpha_1^{(4)} = \sigma_1 \tensor I_2,
\quad
\alpha_2^{(4)} = \sigma_2 \tensor \sigma_1,
\quad
\alpha_3^{(4)} = \sigma_2 \tensor \sigma_2,
\\
\alpha_4^{(4)} = \sigma_2 \tensor \sigma_3,
\quad
\alpha_5^{(4)} = \sigma_3 \tensor I_2
.
\end{gathered}
\end{equation}
We impose 
\begin{equation}\label{eq:matRAii}
	\begin{gathered}
	A \alpha_j^{(4)} B^* = A B^* = 0
	\\
	A \alpha_{5}^{(4)} B^* \neq 0
	\end{gathered}
\end{equation}
for $j \in \{1,2,3,4\}$. Let us denote with $a=(a_1,\dots,a_4)$ and $b=(b_1,\dots,b_4)$ the first rows respectively of $A$ and $B$. Hence from the conditions \eqref{eq:matRAii} we infer
\begin{align*}
a_3 \overline{b_1} + a_4 \overline{b_2} + a_1 \overline{b_3} + a_2 \overline{b_4} &=0
\\
-a_4 \overline{b_1} - a_3 \overline{b_2} + a_2 \overline{b_3} + a_1 \overline{b_4} &=0
\\
a_4 \overline{b_1} - a_3 \overline{b_2} - a_2 \overline{b_3} + a_1 \overline{b_4} &=0
\\
-a_3 \overline{b_1} + a_4 \overline{b_2} + a_1 \overline{b_3} - a_2 \overline{b_4} &=0
\\	
a_1 \overline{b_1} + a_2 \overline{b_2} + a_3 \overline{b_3} + a_4 \overline{b_4} &=0
\\
a_1 \overline{b_1} + a_2 \overline{b_2} - a_3 \overline{b_3} - a_4 \overline{b_4} &\neq0
\end{align*}
and equivalently
\begin{gather*}
a_3 \overline{b_1} + a_2 \overline{b_4} 
=
a_3 \overline{b_2} - a_1 \overline{b_4} 
=
a_4 \overline{b_1} - a_2 \overline{b_3} 
=
a_4 \overline{b_2} + a_1 \overline{b_3} 
=0
\\
a_1 \overline{b_1} + a_2 \overline{b_2} = - a_3 \overline{b_3} - a_4 \overline{b_4} \neq 0.
\end{gather*}
However, this system is impossible. Suppose indeed that $a_1 \neq0$. Then
\begin{gather*}
\overline{b_3} = - \frac{a_4}{a_1} \overline{b_2},
\qquad
\overline{b_4} = \frac{a_3}{a_1} \overline{b_2}
\\
a_4 \left( \overline{b_1} + \frac{a_2}{a_1} \overline{b_2} \right) 
=
a_3 \left( \overline{b_1} + \frac{a_2}{a_1} \overline{b_2} \right) =0
\\
a_1 \overline{b_1} + a_2 \overline{b_2} = - a_3 \overline{b_3} - a_4 \overline{b_4} \neq 0.
\end{gather*}
From the first two lines one infers that or $a_1 \overline{b_1} + a_2 \overline{b_2}=0$, or $a_3=a_4=b_3=b_4=0$. Both the possibilities are incompatible with the last condition. Similarly one can prove that the system is impossible also when $a_1=0$.

Now, let us look at the dimension $n=6$. Here we can build examples by the aid of our brick matrices, but it is not so straightforward as in the odd-dimensional case, and we need to be sneaky. 
Firstly, recall that
\begin{equation*}
	\alpha_1^{(6)} = \sigma_1 \tensor I_4,
	\quad
	\alpha_{k+1}^{(6)} = \sigma_2 \tensor \alpha_k^{(4)},
	\quad
	\alpha_7^{(6)} = \sigma_3 \tensor I_4
\end{equation*}
for $k\in\{1,\dots,5\}$.
%
We search matrices $A,B \in \C^{8\times8}$ such that
\begin{gather*}
AB^* = A\alpha^{(6)}_k B^*=0
\\
B^*A \neq 0 \neq A \alpha^{(6)}_7 B^*
\end{gather*}
for $k\in\{1,\dots,6\}$.
Let us start with the ansatz that $A$ and $B$ have the following structure:
\begin{equation*}
A =
\frac{1}{2}
\begin{pmatrix}
\widetilde{A} & -\widetilde{A} (\sigma_1\tensor\sigma_1)
\\
\widetilde{A}  & -\widetilde{A}(\sigma_1\tensor\sigma_1)
\end{pmatrix},
\qquad
B =
\frac{1}{2}
\begin{pmatrix}
\widetilde{B} & \widetilde{B} (\sigma_1\tensor\sigma_1)
\\
\widetilde{B}  & \widetilde{B}(\sigma_1\tensor\sigma_1)
\end{pmatrix}
\end{equation*}
with $\widetilde{A},\widetilde{B} \in \C^{4\times4}$. In this way, recalling the definition of the Dirac matrices and observing that $(\sigma_1 \tensor \sigma_1)^2=I_4$, the conditions $AB^*=0$ and $A\alpha_1^{(6)}B^*=A(\sigma_1 \tensor I_4)B^*=0$ are immediately verified and the other ones become
\begin{align}\label{n6k}
A \alpha_{k+1}^{(6)} B^* =
-\frac{i}{4}
\begin{pmatrix}
1 & 1
\\
1 & 1
\end{pmatrix}
\tensor
\widetilde{A}[(\sigma_1\tensor\sigma_1)\alpha^{(4)}_k + \alpha^{(4)}_k(\sigma_1\tensor\sigma_1)]\widetilde{B}^* &=0
\end{align}
for $k\in\{1,\dots,5\}$, and
\begin{align*}
A \alpha_7 B^* =
A (\sigma_3 \tensor I_4) B^*=
\frac{1}{2}
\begin{pmatrix}
1 & 1
\\
1 & 1
\end{pmatrix}
\tensor 
\widetilde{A}\widetilde{B}^* &\neq 0
\\
B^*A =
\frac{1}{2}
\begin{pmatrix}
\widetilde{B}^*\widetilde{A} & -\widetilde{B}^*\widetilde{A}(\sigma_1\tensor\sigma_1)
\\
(\sigma_1\tensor\sigma_1)\widetilde{B}^*\widetilde{A} & -(\sigma_1\tensor\sigma_1)\widetilde{B}^*\widetilde{A}(\sigma_1\tensor\sigma_1) 
\end{pmatrix}
&\neq 0.
\end{align*}
In \eqref{n6k}, exploiting the definition of the Dirac matrices in dimension $n=4$, the anticommutation relations \eqref{clifford} and the identities $\sigma_1 \sigma_2=i\sigma_3$ and $\sigma_1\sigma_3=-i\sigma_2$, we get that also the identities
\begin{align*}
A \alpha_{3}^{(6)} B^* = A \alpha_{6}^{(6)} B^* =0
%
\end{align*}
are immediately satisfied, and the remaining ones reduce to
\begin{align*}
A \alpha_{2}^{(6)} B^* = 
-\frac{i}{2}
\begin{pmatrix}
1 & 1
\\
1 & 1
\end{pmatrix}
\tensor
\widetilde{A} (I_2 \tensor \sigma_1) \widetilde{B}^* &=0
\\
A \alpha_{4}^{(6)} B^* = 
\frac{i}{2}
\begin{pmatrix}
1 & 1
\\
1 & 1
\end{pmatrix}
\tensor
\widetilde{A} (\sigma_3 \tensor \sigma_3) \widetilde{B}^* &=0
\\
A \alpha_{5}^{(6)} B^* = 
-\frac{i}{2}
\begin{pmatrix}
1 & 1
\\
1 & 1
\end{pmatrix}
\tensor
\widetilde{A} (\sigma_3 \tensor \sigma_2) \widetilde{B}^* &=0.
\end{align*}
Thus, it would be enough to find $\widetilde{A},\widetilde{B} \in \C^{4\times4}$ such that 
\begin{gather*}
\widetilde{A} (I_2 \tensor \sigma_1) \widetilde{B}^*
=
\widetilde{A} (\sigma_3 \tensor \sigma_3) \widetilde{B}^*
=
\widetilde{A} (\sigma_3 \tensor \sigma_2) \widetilde{B}^*
=0
\\
\widetilde{A}\widetilde{B}^* \neq 0 \neq \widetilde{B}^* \widetilde{A}.
\end{gather*}
This step is easily achieved exploiting our brick matrices, indeed we can choose
\begin{equation*}
	\widetilde{A}_\pm= \rho^3_\pm \tensor \rho^k_\pm,
	\quad
	\widetilde{B}_\pm = \tau^3_\pm \tensor \tau^k_\pm
\end{equation*}
for any fixed $k\in\{0,2,3\}$. In this way we can construct many examples for the $6$-dimensional case. If we choose e.g. $k=3$ in the above definition of $\widetilde{A}$ and $\widetilde{B}$, and taking again in account Lemma\til\ref{lem:rec-pot}, we can exhibit the following couple of examples of matrices satisfying RA(ii) for even dimension $n\ge6$:
\begin{align*}
	V_\pm
	&=
	\frac{1}{8}
	\begin{pmatrix}
		(I_2 \pm \sigma_2)^{\tensor 2} & -(\sigma_1 \mp i \sigma_3)^{\tensor 2}
		\\
		(\sigma_1 \pm i \sigma_3)^{\tensor 2} & -(I_2 \mp \sigma_2)^{\tensor 2}
	\end{pmatrix}
	\tensor 
	I_{N/8}
	\\
	&=
	\frac{1}{8}
	\begin{pmatrix}
		1 & \mp i & \mp i & -1 & 1 & \pm i & \pm i & -1
		\\
		\pm i & 1 & 1 & \mp i & \pm i & -1 & -1 & \mp i
		\\
		\pm i & 1 & 1 & \mp i & \pm i & -1 & -1 & \mp i
		\\
		-1 & \pm i & \pm i & 1 & -1 & \mp i & \mp i & 1
		\\
		-1 & \pm i & \pm i & 1 & -1 & \mp i & \mp i & 1
		\\
		\pm i & 1 & 1 & \mp i & \pm i & -1 & -1 & \mp i
		\\
		\pm i & 1 & 1 & \mp i & \pm i & -1 & -1 & \mp i
		\\
		1 & \mp i & \mp i & -1 & 1 & \pm i & \pm i & -1
	\end{pmatrix}
	\tensor 
	I_{N/8}
	\\
	&=
	B_\pm^* A_\pm
\end{align*}
where
\begin{align*}
A_\pm
&=
\frac{1}{8}
\begin{pmatrix}
(I_2 \pm \sigma_2)^{\tensor 2} & -(\sigma_1 \mp i \sigma_3)^{\tensor 2}
\\
(I_2 \pm \sigma_2)^{\tensor 2} & -(\sigma_1 \mp i \sigma_3)^{\tensor 2}
\end{pmatrix}
\tensor 
I_{N/8},
\\
B_\pm
&=
\frac{1}{8}
\begin{pmatrix}
(I_2 \pm \sigma_2)^{\tensor 2} & (\sigma_1 \mp i \sigma_3)^{\tensor 2}
\\
(I_2 \pm \sigma_2)^{\tensor 2} & (\sigma_1 \mp i \sigma_3)^{\tensor 2}
\end{pmatrix}
\tensor 
I_{N/8}
.
\end{align*}
%

\subsubsection{Case of RA(iii).} Mutatis mutandis, the situation is similar to the the case of RA(ii), hence we skip the computations. As above, one can prove the absence of our desired potentials in dimension $n=2$ and $n=4$. In even dimension $n\ge6$ instead, we impose to $A$ and $B$ to have the structure 
\begin{equation*}
A =
\frac{1}{2}
\begin{pmatrix}
\widetilde{A} & \widetilde{A} (\sigma_1\tensor\sigma_1)
\\
\widetilde{A}  & \widetilde{A}(\sigma_1\tensor\sigma_1)
\end{pmatrix}
\tensor
I_{N/8}
,
\qquad
B =
\frac{1}{2}
\begin{pmatrix}
\widetilde{B} & \widetilde{B} (\sigma_1\tensor\sigma_1)
\\
\widetilde{B}  & \widetilde{B}(\sigma_1\tensor\sigma_1)
\end{pmatrix}
\tensor
I_{N/8}
,
\end{equation*}
where $\widetilde{A}, \widetilde{B} \in \C^{4\times4}$ have to satisfy the relations
\begin{gather*}
	\widetilde{A} (\sigma_1 \tensor \sigma_1) \widetilde{B}^*
	=
	\widetilde{A} (\sigma_3 \tensor I_2) \widetilde{B}^*
	=
	\widetilde{A} (\sigma_2 \tensor \sigma_1) \widetilde{B}^*
	=0
	\\
	\widetilde{A}\widetilde{B}^* \neq 0 \neq \widetilde{B}^* \widetilde{A}.
\end{gather*}
For example we can choose again
\begin{equation*}
	\widetilde{A}_\pm = \rho^3_\pm \tensor \rho^3_\pm,
	\quad
	\widetilde{B}_\pm = \tau^3_\pm \tensor \tau^3_\pm,
\end{equation*}
and hence we obtain the following couple of examples of matrices satisfying RA(iii) in even dimension $n\ge6$:
\begin{align*}
V_\pm
&=
\frac{1}{8}
\begin{pmatrix}
(I_2 \pm \sigma_2)^{\tensor 2} & (\sigma_1 \mp i \sigma_3)^{\tensor 2}
\\
(\sigma_1 \pm i \sigma_3)^{\tensor 2} & (I_2 \mp \sigma_2)^{\tensor 2}
\end{pmatrix}
\tensor 
I_{N/8}
\\
&=
\frac{1}{8}
\begin{pmatrix}
1 & \mp i & \mp i & -1 & -1 & \mp i & \mp i & 1
\\
\pm i & 1 & 1 & \mp i & \mp i & 1 & 1 & \pm i
\\
\pm i & 1 & 1 & \mp i & \mp i & 1 & 1 & \pm i
\\
-1 & \pm i & \pm i & 1 & 1 & \pm i & \pm i & -1
\\
-1 & \pm i & \pm i & 1 & 1 & \pm i & \pm i & -1
\\
\pm i & 1 & 1 & \mp i & \mp i & 1 & 1 & \pm i
\\
\pm i & 1 & 1 & \mp i & \mp i & 1 & 1 & \pm i
\\
1 & \mp i & \mp i & -1 & -1 & \mp i & \mp i & 1
\end{pmatrix}
\tensor 
I_{N/8}
\\
&=
B_\pm^* A_\pm
\end{align*}
where
\begin{align*}
A_\pm
=
\frac{1}{8}
\begin{pmatrix}
(I_2 \pm \sigma_2)^{\tensor 2} & (\sigma_1 \mp i \sigma_3)^{\tensor 2}
\\
(I_2 \pm \sigma_2)^{\tensor 2} & (\sigma_1 \mp i \sigma_3)^{\tensor 2}
\end{pmatrix}
\tensor 
I_{N/8}
=
B_\pm
.
\end{align*}

\subsubsection{Case of RA(iv)} 
In dimension $n=2$ there are no potentials, and this can be easily seen as in the above case of RA(ii). In even dimension $n\ge4$ instead, recalling the definition of the Dirac matrices in $4$-dimension \eqref{eq:dirac4} and Lemma\til\ref{lem:rec-pot},
it is easy to check that a couple of examples for our desired matrices is obtained choosing $V_\pm = B_\pm^* A_\pm$ with
\begin{align*}
	A_\pm = \rho^2_\pm \tensor \rho^0_\pm \tensor I_{N/4},
	\qquad
	B_\pm = \tau^2_\pm \tensor \tau^0_\pm \tensor I_{N/4},
\end{align*}
hence videlicet
\begin{align*}
	V_\pm
	=
	\frac{i}{4}
	\begin{pmatrix}
	1 \pm 1 & -1 \mp 1 & 0 & 0
	\\
	1 \pm 1 & -1 \mp 1 & 0 & 0
	\\
	0 & 0 & -1 \pm 1 & -1 \pm 1 
	\\
	0 & 0 & 1 \mp 1 & 1 \mp 1
	\end{pmatrix}
	\tensor
	I_{N/4}.
\end{align*}

This concludes the parade of examples for the matricial parts $V$ of the potentials satisfying our Rigidity Assumptions (i)--(iv).


\section*{Acknowledgement}
The first author is partially supported by JSPS KAKENHI Grant-in-Aid for Young Scientists (B) \#JP17K14218 and Grant-in-Aid for Scientific Research (B) \#JP17H02854. The second author is  member of the \lq\lq Gruppo Nazionale per L'Analisi Matematica, la Probabilit\`{a} e le loro Applicazioni'' (GNAMPA) of the \lq\lq Istituto Nazionale di Alta Matematica'' (INdAM) and he is partially supported by \textit{Progetti per Avvio alla Ricerca di Tipo 1} and \textit{Progetti di Mobilità all'Estero per Dottorandi} by Sapienza Universit\`{a} di Roma.
He is also grateful to the first author and to Osaka University for the hearty hospitality during his stay there (December 2020 -- May 2021), when the present work was mostly developed.


\bibliographystyle{abbrv}
\bibliography{ref}

\begin{thebibliography}{10}

\bibitem{AbramovAslanyanDavies01}
A.~A. Abramov, A.~Aslanyan, and E.~B. Davies.
\newblock Bounds on complex eigenvalues and resonances.
\newblock {\em J. Phys. A}, 34(1):57--72, 2001.

\bibitem{BarceloRuizVega97}
J.~A. Barcelo, A.~Ruiz, and L.~Vega.
\newblock Weighted estimates for the {H}elmholtz equation and some
  applications.
\newblock {\em J. Funct. Anal.}, 150(2):356--382, 1997.

\bibitem{Bogli17}
S.~B\"{o}gli.
\newblock Schr\"{o}dinger operator with non-zero accumulation points of complex
  eigenvalues.
\newblock {\em Comm. Math. Phys.}, 352(2):629--639, 2017.

\bibitem{CassanoIbrogimovKrejcirikStampach20}
B.~Cassano, O.~O. Ibrogimov, D.~Krej\v{c}i\v{r}\'{\i}k, and F.~\v{S}tampach.
\newblock Location of eigenvalues of non-self-adjoint discrete {D}irac
  operators.
\newblock {\em Ann. Henri Poincar\'{e}}, 21(7):2193--2217, 2020.

\bibitem{CassanoPizzichilloVega20}
B.~Cassano, F.~Pizzichillo, and L.~Vega.
\newblock A {H}ardy-type inequality and some spectral characterizations for the
  {D}irac-{C}oulomb operator.
\newblock {\em Rev. Mat. Complut.}, 33(1):1--18, 2020.

\bibitem{Cossetti17}
L.~Cossetti.
\newblock Uniform resolvent estimates and absence of eigenvalues for {L}am\'{e}
  operators with complex potentials.
\newblock {\em J. Math. Anal. Appl.}, 455(1):336--360, 2017.

\bibitem{CossettiFanelliKrejcirik20}
L.~Cossetti, L.~Fanelli, and D.~Krej\v{c}i\v{r}\'{\i}k.
\newblock Absence of eigenvalues of {D}irac and {P}auli {H}amiltonians via the
  method of multipliers.
\newblock {\em Comm. Math. Phys.}, 379(2):633--691, 2020.

\bibitem{CossettiKrejcirik20}
L.~Cossetti and D.~Krej\v{c}i\v{r}\'{\i}k.
\newblock Absence of eigenvalues of non-self-adjoint {R}obin {L}aplacians on
  the half-space.
\newblock {\em Proc. Lond. Math. Soc. (3)}, 121(3):584--616, 2020.

\bibitem{Cuenin14}
J.-C. Cuenin.
\newblock Estimates on complex eigenvalues for {D}irac operators on the
  half-line.
\newblock {\em Integral Equations Operator Theory}, 79(3):377--388, 2014.

\bibitem{Cuenin17}
J.-C. Cuenin.
\newblock Eigenvalue bounds for {D}irac and fractional {S}chr\"{o}dinger
  operators with complex potentials.
\newblock {\em J. Funct. Anal.}, 272(7):2987--3018, 2017.

\bibitem{Cuenin20}
J.-C. Cuenin.
\newblock Improved eigenvalue bounds for {S}chr\"{o}dinger operators with
  slowly decaying potentials.
\newblock {\em Comm. Math. Phys.}, 376(3):2147--2160, 2020.

\bibitem{CueninLaptevTretter14}
J.-C. Cuenin, A.~Laptev, and C.~Tretter.
\newblock Eigenvalue estimates for non-selfadjoint {D}irac operators on the
  real line.
\newblock {\em Ann. Henri Poincar\'{e}}, 15(4):707--736, 2014.

\bibitem{CueninSiegl18}
J.-C. Cuenin and P.~Siegl.
\newblock Eigenvalues of one-dimensional non-self-adjoint {D}irac operators and
  applications.
\newblock {\em Lett. Math. Phys.}, 108(7):1757--1778, 2018.

\bibitem{CueninTretter16}
J.-C. Cuenin and C.~Tretter.
\newblock Non-symmetric perturbations of self-adjoint operators.
\newblock {\em J. Math. Anal. Appl.}, 441(1):235--258, 2016.

\bibitem{DaviesNath02}
E.~B. Davies and J.~Nath.
\newblock Schr\"{o}dinger operators with slowly decaying potentials.
\newblock volume 148, pages 1--28. 2002.
\newblock On the occasion of the 65th birthday of Professor Michael Eastham.

\bibitem{Dubuisson14}
C.~Dubuisson.
\newblock On quantitative bounds on eigenvalues of a complex perturbation of a
  {D}irac operator.
\newblock {\em Integral Equations Operator Theory}, 78(2):249--269, 2014.

\bibitem{DAnconaFanelliKrejcirikSchiavone21}
P.~D’Ancona, L.~Fanelli, D.~Krej\v{c}i\v{r}\'{\i}k, and N.~M. Schiavone.
\newblock Localization of eigenvalues for non-self-adjoint dirac and
  klein-gordon operators.
\newblock {\em arXiv preprint arXiv:2104.13647}, 2021.

\bibitem{DAnconaFanelliSchiavone21}
P.~D’Ancona, L.~Fanelli, and N.~M. Schiavone.
\newblock Eigenvalue bounds for non-selfadjoint {D}irac operators.
\newblock {\em Math. Ann.}, pages 1--24, 2021.

\bibitem{Enblom16}
A.~Enblom.
\newblock Estimates for eigenvalues of {S}chr\"{o}dinger operators with
  complex-valued potentials.
\newblock {\em Lett. Math. Phys.}, 106(2):197--220, 2016.

\bibitem{Enblom18}
A.~Enblom.
\newblock Resolvent estimates and bounds on eigenvalues for {D}irac operators
  on the half-line.
\newblock {\em J. Phys. A}, 51(16):165203, 13, 2018.

\bibitem{ErdoganGoldbergGreen19}
M.~B. Erdo\u{g}an, M.~Goldberg, and W.~R. Green.
\newblock Limiting absorption principle and {S}trichartz estimates for {D}irac
  operators in two and higher dimensions.
\newblock {\em Comm. Math. Phys.}, 367(1):241--263, 2019.

\bibitem{ErdoganGreen21}
M.~B. Erdo\u{g}an and W.~R. Green.
\newblock On the one dimensional {D}irac equation with potential.
\newblock {\em J. Math. Pures Appl. (9)}, 151:132--170, 2021.

\bibitem{FanelliKrejcirik19}
L.~Fanelli and D.~Krej\v{c}i\v{r}\'{\i}k.
\newblock Location of eigenvalues of three-dimensional non-self-adjoint {D}irac
  operators.
\newblock {\em Lett. Math. Phys.}, 109(7):1473--1485, 2019.

\bibitem{FanelliKrejcirikVega18-JFA}
L.~Fanelli, D.~Krej\v{c}i\v{r}\'{\i}k, and L.~Vega.
\newblock Absence of eigenvalues of two-dimensional magnetic {S}chr\"{o}dinger
  operators.
\newblock {\em J. Funct. Anal.}, 275(9):2453--2472, 2018.

\bibitem{FanelliKrejcirikVega18-JST}
L.~Fanelli, D.~Krej\v{c}i\v{r}\'{\i}k, and L.~Vega.
\newblock Spectral stability of {S}chr\"{o}dinger operators with subordinated
  complex potentials.
\newblock {\em J. Spectr. Theory}, 8(2):575--604, 2018.

\bibitem{Frank11}
R.~L. Frank.
\newblock Eigenvalue bounds for {S}chr\"{o}dinger operators with complex
  potentials.
\newblock {\em Bull. Lond. Math. Soc.}, 43(4):745--750, 2011.

\bibitem{Frank18}
R.~L. Frank.
\newblock Eigenvalue bounds for {S}chr\"{o}dinger operators with complex
  potentials. {III}.
\newblock {\em Trans. Amer. Math. Soc.}, 370(1):219--240, 2018.

\bibitem{FrankLaptevSeiringer06}
R.~L. Frank, A.~Laptev, E.~H. Lieb, and R.~Seiringer.
\newblock Lieb-{T}hirring inequalities for {S}chr\"{o}dinger operators with
  complex-valued potentials.
\newblock {\em Lett. Math. Phys.}, 77(3):309--316, 2006.

\bibitem{FrankSabin17}
R.~L. Frank and J.~Sabin.
\newblock Restriction theorems for orthonormal functions, {S}trichartz
  inequalities, and uniform {S}obolev estimates.
\newblock {\em Amer. J. Math.}, 139(6):1649--1691, 2017.

\bibitem{FrankSimon17}
R.~L. Frank and B.~Simon.
\newblock Eigenvalue bounds for {S}chr\"{o}dinger operators with complex
  potentials. {II}.
\newblock {\em J. Spectr. Theory}, 7(3):633--658, 2017.

\bibitem{Gutierrez04}
S.~Guti\'{e}rrez.
\newblock Non trivial {$L^q$} solutions to the {G}inzburg-{L}andau equation.
\newblock {\em Math. Ann.}, 328(1-2):1--25, 2004.

\bibitem{HansmannKrejcirik2020}
M.~Hansmann and D.~Krej\v{c}i\v{r}\'{\i}k.
\newblock The abstract {B}irman-{S}chwinger principle and spectral stability.
\newblock {\em arXiv preprint arXiv:2010.15102}, 2020.

\bibitem{KalfYamada01}
H.~Kalf and O.~Yamada.
\newblock Essential self-adjointness of {$n$}-dimensional {D}irac operators
  with a variable mass term.
\newblock {\em J. Math. Phys.}, 42(6):2667--2676, 2001.

\bibitem{KenigRuizSogge87}
C.~E. Kenig, A.~Ruiz, and C.~D. Sogge.
\newblock Uniform {S}obolev inequalities and unique continuation for second
  order constant coefficient differential operators.
\newblock {\em Duke Math. J.}, 55(2):329--347, 1987.

\bibitem{KwonLee20}
Y.~Kwon and S.~Lee.
\newblock Sharp resolvent estimates outside of the uniform boundedness range.
\newblock {\em Comm. Math. Phys.}, 374(3):1417--1467, 2020.

\bibitem{LaptevSafronov09}
A.~Laptev and O.~Safronov.
\newblock Eigenvalue estimates for {S}chr\"{o}dinger operators with complex
  potentials.
\newblock {\em Comm. Math. Phys.}, 292(1):29--54, 2009.

\bibitem{LeeSeo19}
Y.~Lee and I.~Seo.
\newblock A note on eigenvalue bounds for {S}chr\"{o}dinger operators.
\newblock {\em J. Math. Anal. Appl.}, 470(1):340--347, 2019.

\bibitem{LiebSeiringer-book}
E.~H. Lieb and R.~Seiringer.
\newblock {\em The stability of matter in quantum mechanics}.
\newblock Cambridge University Press, Cambridge, 2010.

\bibitem{RenXiZhang18}
T.~Ren, Y.~Xi, and C.~Zhang.
\newblock An endpoint version of uniform {S}obolev inequalities.
\newblock {\em Forum Math.}, 30(5):1279--1289, 2018.

\bibitem{Safronov10}
O.~Safronov.
\newblock Estimates for eigenvalues of the {S}chr\"{o}dinger operator with a
  complex potential.
\newblock {\em Bull. Lond. Math. Soc.}, 42(3):452--456, 2010.

\end{thebibliography}

\end{document}